\documentclass[11pt]{article}
\usepackage{fullpage}
\usepackage{algorithmic}
\usepackage{algorithm}
\usepackage{epsfig}
\usepackage{graphicx}
\usepackage{latexsym}
\usepackage{amsmath}
\usepackage{amsfonts}
\usepackage{multirow}
\usepackage{amssymb}
\usepackage[sort,numbers]{natbib}
 \usepackage{tikz}

\usetikzlibrary{shapes,arrows,shadows}
\usepackage{amsmath,bm,times}

\usepackage{mathrsfs}
\usepackage{pifont}
\usepackage{yhmath}
\usepackage{bbm}
\usepackage{amsthm}
\usepackage{booktabs}
\usepackage{stmaryrd}
\usepackage{wasysym}
\usepackage{bbm}
\usepackage{array}
\usepackage{enumerate}
\usepackage{dsfont}

\numberwithin{equation}{section}
\numberwithin{figure}{section}

\newtheorem{theorem}{Theorem}[section]
\newtheorem{lemma}[theorem]{Lemma}

\newtheorem{corollary}[theorem]{Corollary}
\newtheorem{proposition}[theorem]{Proposition}
\newtheorem{definition}[theorem]{Definition}
\newtheorem{remark}[theorem]{Remark}

\newtheorem{observation}[theorem]{Observation}

\theoremstyle{remark}

\makeatletter
\newcommand\figcaption{\def\@captype{figure}\caption}
\newcommand\tabcaption{\def\@captype{table}\caption}
\makeatother

\DeclareMathAlphabet{\mathpzc}{OT1}{pzc}{m}{it}

\begin{document}
\newcounter{my}
\newenvironment{mylabel}
{
\begin{list}{(\roman{my})}{
\setlength{\parsep}{-1mm}
\setlength{\labelwidth}{8mm}
\usecounter{my}}
}{\end{list}}

\newcounter{my2}
\newenvironment{mylabel2}
{
\begin{list}{(\alph{my2})}{
\setlength{\parsep}{-0mm} \setlength{\labelwidth}{8mm}
\setlength{\leftmargin}{3mm}
\usecounter{my2}}
}{\end{list}}

\newcounter{my3}
\newenvironment{mylabel3}
{
\begin{list}{(\alph{my3})}{
\setlength{\parsep}{-1mm}
\setlength{\labelwidth}{8mm}
\setlength{\leftmargin}{10mm}
\usecounter{my3}}
}{\end{list}}

\title{\bf The local Poincar\'e inequality of stochastic dynamic and application to the Ising model \thanks{This work is supported by the National Natural Science Foundation of China (No. 12288201).}}
\author{Cui Kaiyuan\thanks{Institute of Applied Mathematics, Academy of Mathematics and Systems Science,
		Chinese Academy of Sciences, Beijing 100190, China,({ cuiky@amss.ac.cn})}\and Gong Fuzhou \thanks{Institute of Applied Mathematics, Academy of Mathematics and Systems Science,
		Chinese Academy of Sciences, Beijing 100190, China,({ fzgong@amt.ac.cn})}}
	\date{}
\maketitle

\vspace{-5em}

\begin{center}\large

\end{center}


\begin{abstract}
Inspired by the idea of stochastic quantization proposed by Parisi and Wu, we construct the transition probability matrix which plays a central role in the renormalization group through a stochastic differential equation. By establishing the discrete time stochastic dynamics, the renormalization procedure can be characterized from the perspective of probability. Hence, we will focus on the  investigation of the infinite dimensional stochastic dynamic.
From the stochastic point of view, the discrete time stochastic dynamic can induce a  Markov chain. Via calculating the square field operator and the Bakry-\'Emery curvature for a class of two-points functions, the local Poincar\'e inequality is established, from which the estimate of correlation functions can also be obtained. Finally, under the condition of ergodicity, by choosing the couple relationship between the system parameter $K$ and the system time $T$ properly when $T\rightarrow +\infty$, the two-points correlation functions for limit system are also estimated. 	\\

\end{abstract}
\section{Introduction}\label{sec:1}
\subsection{Motivation}\label{sec:1.1}
The subject of critical phenomena and phase transitions has fascinated mathematicians for over half a century; see
\cite{Roland2019book}. In the physics literatures, critical phenomena are understood via the renormalisation group(RG) method. After the modern reconstruction of the renormalisation group procedure by Wilson in 1971, the renormalization group method  revolutionized the mathematical sciences and has appeared in thousands papers devoted to the development of the understanding of physical, social, biological and financial systems. As Kadanoff put it,  Wilson has helped add our profession a new style of work and of thinking; see \cite{Kadanoff2013document}. In view of  these wonderful accomplishments, Wilson received the 1982 Nobel Prize within physics. Nowadays, the renormalization has been not only a powerful technical tool  but also a primary insight in physics which provides  connections between the behavior at one scale and the phenomena at very different scale. 

Since then, investigators have made many efforts to elucidate the nature and universal properties of the  renormalization group flows between them; see \cite{Cardy1996book}. The pioneering work, called the $c$ -theorem, claims that there exists a microscopic irreversible process along the renormalization group flow, which was owing to Zamolodchikov; see \cite{Kim2017document}. Ever since this observation, there are a series of papers published in the contex of renormalization. For example, the renormalization group has interesting features as a system of autonomous ordinary differential equations in coupling parameters space. However, this is a infinite-dimensional equation which should be truncated. Based on the assumption that the infinite number of
neglected irrelevant couplings produce some uncertainty in the values of the preserved
couplings,  Gaite advanced a stochastic formulation of the renormalization group in \cite{Gaite2004document} through adding a noise to the RG equations; see \cite{Gaite1996document,Casini2007document,Apenko2012document,Gordon2021document} for more instances and different approaches in this direction.  

Specifically, the most known example in the field of \emph{Real Space Renormalization Group} is the classic one-dimensional Ising model with  periodic boundary condition, which can be solved exactly. The basic idea lying at the heart of \emph{Real Space Renormalization Group} is  the ``coarse-graining" procedure which was formulated by Kaddanoff in 1966. It states that, no matter how many times the blocking transformation is iterated, the dominant interactions will be short-ranged; see \cite{Cardy1996book}. More precisely, in concrete implementation, the blocking transformation maps the  original ``spins''  to a sequence of new ``block spins'' but leaves the structure of probability distributions of quantities which depend on spins invariant. In the Markov words, there is a probability transition function connects the original ``spins'' measure and new ``block spins'' measure, and the construction of this probability transition function plays a central role in the ``coarse-graining" procedure,  which will be discussed more carefully in the next section.

One of the most routinely important features in physical literature is that the correctness of a new idea is best borne out by testing in exactly solved models. Inspired by the idea of stochastic quantization, which is proposed by Giorgio Parisi, who is  the Nobel Prize winner in physics  at 2021,  and Wu, we  propose a new perspective to reconstruct the renormalization procedure, from the stochastic dynamic point of view. In this formulation, a stochastic difference equation will be investigated through  the local Poincar\'e  inequality and two-points correlation functions.

Poincar\'e  inequality is a powerful tool in investigating statistical mechanics models. On the one hand, it has been recognized that the concentration of measure phenomenon has spread out to an impressively wide range of illustrations and applications, including  probability theory, statistical mechanics, random matrix theory, random graph models, stochastic dynamics, and so on.
Roughly speaking, the concentration inequality implies a sharp estimation of a random variable by bounding the deviation from it's expectation with an associated probability
; see \cite{Stephane2013book,Chung2006document} for details. 

As is well known, the Poincar\'e inequality can imply the exponential concentration in the scalar setting, which is proved by Gromov and Milman on Riemannian manifolds  \cite{Gromov1983document}. During the past decades, in the context of random matrix,  researchers take serious effort to establish 
the concentration theory for random matrix models; see \cite{Ahlswede2003document,Mackey2014document,Paulin2016document,Tropp2011document,Tropp2021document}. Since then, one natural idea is to develop matrix versions of techniques that connect the  concentration to functional inequalities, even though it will meet new difficulties because of noncommutativity. Very recently, from the perspective of functional inequalities, Aoun et.al firstly show that the Poincar\'e inequality in matrix setting can also lead to subexponential concentration of a random matrix under the $\ell^{2}$  form in paper \cite{Aoun2020document}. Ever since this pioneering paper of Aoun, Tropp and Huang show us another concise argument using the symmetrization  technique and matrix chain rule inequality. Simultaneously, they demonstrate that local Poincar\'e  inequalities, which are equivalent to the scalar Bakry-\'Emery criterion, lead to the optimal subgaussian concentration results; see \cite{Tropp2021document,Huang2021document}.

On the other hand,  the Poincar\'e inequality implies the two-points correlation functions estimation. And the two-points correlation functions play an important role in investigation of statistic physical models, especially in characterizing the properties near the critical points. From the perspective of probability, the information contained in the moments of Gibbs measure with light-tailed distribution,  to some extent, can determine the measure. Briefly speaking, under the circumstance of Gaussian distribution, according to Feynman graph calculus and Wick theorem, arbitrary finite-points correlation functions can be expressed as the polynomial of  the two-points correlation functions. Things will be complex for general measures with non-trivial potential terms, fortunately, it can also be estimated by the two-points correlation functions; see  \cite{Glimm1987book}. Hence, the two-points correlation functions play a crucial role in analyzing statistical mechanic systems,
for example, Cédric Bény and Tobias J. Osborne recently introduce a thermodynamic quantity which decreases along the renormalization flow  and can be expressed by two-points correlation functions to investigate the behavior of renormalization flow; see \cite{Beny2015document}.

Observing the insight relationship between the Poincar\'e inequality  and the properties near the critical points of statistical mechanic systems, mathematicians  attempt to develop techniques to analysis a class of spin systems, such as stochastic Ising model, random field Ising model. Initial efforts in this direction were due to Stroock, Zegarlinski, Yoshida, Yau H.T and so on; see \cite{ZEGARLISKI1992document,Lu1993document,Zegarlinski1996document,STROOCK1992document,STROOCK1992document1,STROOCK1992document2,YOSHIDA2001document}. Recently, under the condition that the operator norm of the interaction matrix is smaller than 1, Eldan et.al \cite{Eldan2022document} establish a Poincar\'e inequality for Ising model with general quadratic interactions and  the Sherrington-Kirkpatrick (SK) model in the fields of spin glass. Before this result, Roland Bauerschmidt and Thierry Bodineau has proved that the SK model satisfies a Log-Sobolev inequality, with another Dirichlet form, at sufficiently high temperature;
see \cite{Eldan2022document,Roland2019document}.

Indeed, it is a fascinating topic to  search for methods that can help us establish local  Poincar\'e inequality. Particularly, for discrete configuration space, the generator of 
a Markov process can be seen as a discrete  Laplacian on weighted graphs, which is analogous  to the continuous Laplace operator in  Riemannian geometry; see \cite{Chung1997book}. Even through the discrete  Laplacian is quite different from the continuous one in many aspects, there actually are some similar basic concepts and properties between the graphs and manifolds, such as Ricci curvature, Bakry-\'Emery notation, and so on. The first definition of Ricci curvature on graphs was introduced by Fan Chung and Yau in 1996 \cite{Chung1996document}. Since then, Chung, Lin, Yau and their coauthers published a series of articles in this direction. For example, they proved that the Ricci curvature for a local finite graph or Ricci flat graph can be bounded below, and for graphs with non-negative Ricci curvature, they established some functional inequalities such as Harnack inequalities, we refer to \cite{Chung2000document,Lin2010document,Lin2011document,Chung2014document} for details on these results.
\subsection{Our results}\label{sec:1.2}
For the setup of Markov chain defined on infinite graphs, even if the thermodynamic quantity $f$ we considered only involves finite spatial points, the 
essential computation of the generator and the carr\'e du champ operator $\Gamma$ with respect to  semi-group $\mathcal{P}_{t}$ during the process of deducing the local Poincar\'e inequality is enlarging at each tick of the time because the system is dynamically evolving, the similar problem has occured in \cite{Chung2006document}. Typically, in order to develop a suitable calculus on the carr\'e du champ operator, it is necessary to deal with expressions such as $\Gamma(F,F)$ and $\Gamma_{2}(F,F)$ for any $F$ in some algebra $\mathcal{A}_{0}$, where $\Gamma$ and $\Gamma_{2}$ are defined in Section \ref{sec:3.1}.
By this way, the key point of the calculation will finally be reduced to estimating the Fiedler value of one arbitrary dimensional square matrix, which is complex and intractable, and  the details will be put in Observation \ref{obs:1} in appendix. 

Hence, we can not use the $\Gamma$ calculus directly, owing to the fact that the model and the configuration space $\{+1,-1\}^{\mathbb{Z}}$ we considered does not have special properties, such as local finite, or Ricci flat. However, as we have talked above, the thermodynamic quantities who play a crucial role in physical systems are two-points correlation functions,
and the quantities we want to estimate are two-points functions $f$ rather than the  $F(t,\eta)=\mathcal{P}_{T-t}f(\eta)$. That is to say,
it is not an essential problem but only a technique difficulty. Owing to this observation, we rewrite the $\Gamma$ calculus in matrix form, and expressing the $\Gamma$ and $\Gamma_{2}$ as the quadratic form of $f$ itself, which naturally avoids this trouble even if leads to other difficulties on calculation. The plane of the paper is the following:

In Section 2 we formulate an alternative way to characterize a similar renormalization  procedure of one-dimensional Ising model from the perspective of stochastic quantization. Borrowing the idea from physicists, the  coarse-graining is described by a  stochastic dynamic defined on configuration space $E=\{+ 1,-1\}^{\mathbb{Z}}$.

In Section 3 we extend the classic Bakry-\'Emery criterion to the case of infinite discrete configuration space. For a class of important thermodynamic quantities, two-points correlation functions, the local Poincar\'e inequality at finite time has been established.

In Section 4, using the results we get in section \ref{sec:3}, the stochastic dynamic we formulate in section \ref{sec:2}, which emerges from the  real space renormalization procedure, has been investigated. More precisely, we build the local Poincar\'e inequality and the estimates of correlation functions with respect to two-points functions.

\section{Our model}\label{sec:2}
As we have mentioned in Section \ref{sec:1},  the key issue in establishing a renormalization  procedure is to constructing the transition probability function between the original ``spins'' and new ``block spins''. More concretely, for one-dimensional Ising model with the following reduced Hamiltonian $\mathcal{H}(\sigma)$ 
\begin{align}
	\mathcal{H}(\sigma)=K\sum_{i}\sigma_{i}\sigma_{i+1}+h\sum_{i}\sigma_{i}, 
\end{align}
where $K=\beta J$ is the coupling constant, $h=\beta H$ is external field, which absorb the factor of $\beta=1/ k_{B}T$, and $\sigma_{i}$ represents the spins of the system. Besides, Hamiltonian $\mathcal{H}(\sigma)$ satisfies $$\sum_{\{\sigma\}}\mathcal{H}(\sigma)=0.$$
The main idea of the ``coarse-graining" procedure is to transfering the original spins $\{\sigma\}$ into the new block spin variables $\{S_{I}\}$ but with the structure of Hamiltonian unchanged. That is to say, the new block spin Hamiltonian $\tilde{\mathcal{H}}(S)$ still keeps the form
\begin{align}
	\tilde{\mathcal{H}}(S)=\tilde{K}\sum_{I}S_{I}S_{I+1}+\tilde{h}\sum_{I}S_{I},
\end{align}
Nevertheless,the spatial scale of spin system will be different from the original one.
Expressing the statement above more explicitly in probability words, we need introduce a weight factor 
$P(\sigma;S)$ which satisfies
\begin{equation*}
	P(\sigma;S)\geq 0,\
	\sum_{\{S\}}P(\sigma;S)=1.
\end{equation*}
The new block spin Hamiltonian $\tilde{\mathcal{H}}(S)$ is defined by
\begin{align}\label{1.1}
	\sum_{\sigma}P(\sigma;S)e^{\mathcal{H}(\sigma)}=e^{G_{0}+\tilde{\mathcal{H}}(S)},
\end{align}
where $G_{0}$ is a constant determined by condition  $$\sum_{S}\tilde{\mathcal{H}}(S)=0.$$
Summing over new block spins $S$ to both sides of equality~\eqref{1.1}~, we know that
\begin{align*}
	\sum_{S}\sum_{\sigma}P(\sigma;S)e^{\mathcal{H}(\sigma)}=\sum_{\sigma}e^{\mathcal{H}(\sigma)}\sum_{S}P(\sigma;S)=\sum_{\sigma}e^{\mathcal{H}(\sigma)}=Z,
\end{align*}
which means the partition functions for the original system and the blocked system are the same.
For example, in Figure \ref{Fig:1},  the new block spins are determined by 
\begin{align}\label{1.4}
	S_{I}=sign\{\sigma^{I}_{1}+\sigma^{I}_{2}+\sigma^{I}_{3}\}.
\end{align}
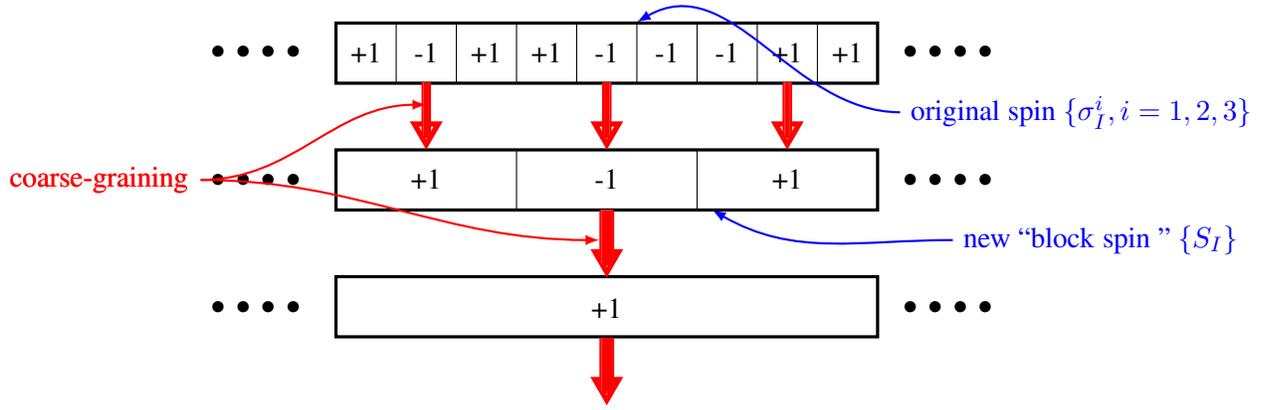
\begin{figure}
	\usetikzlibrary{arrows, decorations.markings}
	
	\tikzstyle{vecArrow} = [thick, decoration={markings,mark=at position
		1 with {\arrow[thick,innerred]{open triangle 60}}},
	double distance=4pt, shorten >= 5.5pt,
	preaction = {decorate},
	postaction = {draw,line width=4pt,red,shorten >= 4.5pt}]
	\tikzstyle{innerblue} = [thick, red,line width=4pt, shorten >= 4.5pt]
	
	\tikzstyle{VecArrow} = [thick, decoration={markings,mark=at position
		1 with {\arrow[thick,innerred]{open triangle 60}}},
	double distance=2pt, shorten >= 5.5pt,
	preaction = {decorate},
	postaction = {draw,line width=2pt,red,shorten >= 4.5pt}]
	\tikzstyle{innerred} = [thick, red,line width=2pt, shorten >= 4.5pt]
	
	\begin{tikzpicture}[set style={{help lines}+=[dashed]}]
		\def\V{-3}   
		\def\EG{3}    

		\def\DZCE{-3.5}  
		\def\LZCE{1}  
		\def\PN{5}   
		
		\pgfmathsetmacro\EC{\EG+\V};
		\pgfmathsetmacro\FZCE{\DZCE+\LZCE};
		
		\foreach \x in {1,2,...,4}
		\draw ({\FZCE+0.6+\x/3},{\EC+0.4}) node {$\bullet$}; 
		\foreach \x in {1,2,...,4}
		\draw ({\FZCE+0.6+\x/3},{\EC+2.1}) node {$\bullet$}; 
		\foreach \x in {1,2,...,4}
		\draw ({\FZCE+0.6+\x/3},{\EC-1.3}) node {$\bullet$}; 
		
		\foreach \x in {1,2,...,4}
		\draw ({\FZCE+9.8+\x/3},{\EC+0.4}) node {$\bullet$}; 
		\foreach \x in {1,2,...,4}
		\draw ({\FZCE+9.8+\x/3},{\EC+2.1}) node {$\bullet$}; 
		\foreach \x in {1,2,...,4}
		\draw ({\FZCE+9.8+\x/3},{\EC-1.3}) node {$\bullet$}; 
		
		\begin{scope}[
			yshift=48,every node/.append style={
				yslant=0,xslant=0},yslant=0,xslant=0
			]
			\fill[white,fill opacity=.9] (0,0) rectangle (7.2,0.8);
			\draw[black,very thick] (0,0) rectangle (7.2,0.8);
			\draw[step=8mm, black] (0,0) grid (7.2,0.8);
		\end{scope}  
		\node  at  (0.4,2.1) {+1};
		\node  at  (1.2,2.1) {-1};
		\node  at  (2,2.1) {+1};
		\node  at  (2.8,2.1) {+1};
		\node  at  (3.6,2.1) {-1};
		\node  at  (4.4,2.1) {-1};
		\node  at  (5.2,2.1) {-1};
		\node  at  (6,2.1) {+1};
		\node  at  (6.8,2.1) {+1};
		\begin{scope}[
			yshift=0,every node/.append style={
				yslant=0,xslant=0},yslant=0,xslant=0
			]
			\fill[white,fill opacity=.9] (0,0) rectangle (7.2,0.8);
			\draw[black,very thick] (0,0) rectangle (7.2,0.8);
			\draw[step=24mm, black] (0,0) grid (7.2,0.8);
		\end{scope}  
		\node  at  (1.2,0.42) {+1};
		\node  at  (3.6,0.42) {-1};
		\node  at  (6,0.42) {+1};
		
		\begin{scope}[
			yshift=-48,every node/.append style={
				yslant=0,xslant=0},yslant=0,xslant=0
			]
			\fill[white,fill opacity=.9] (0,0) rectangle (7.2,0.8);
			\draw[black,very thick] (0,0) rectangle (7.2,0.8);
			\draw[step=72mm, black] (0,0) grid (7.2,0.8);
		\end{scope}  
		
		\node  at  (3.6,-1.3) {+1};

		\draw[red,vecArrow]  (3.6,0) -- (3.6,-0.9);

		\draw[red,VecArrow] (1.2,1.7) -- (1.2,0.8) ;
		\draw[red,VecArrow] (3.6,1.7) -- (3.6,0.8);
		\draw[red,VecArrow] (6,1.7) -- (6,0.8);
		
		\draw[red,vecArrow]   (3.6,-1.7) -- (3.6,-2.6);
		
		\draw[-latex,thick,blue](8.2,-0.4)node[right]{new ``block spin " $\{S_{I}\}$}
		to[out=180,in=330] (5,0);
		\draw[-latex,thick,blue](7.5,1.3)node[right]{original spin $\{\sigma^{i}_{I},i=1,2,3\}$}
		to[out=180,in=30] (4,2.5);
		\draw[-latex,thick,red](-1.8,0.4)node[left]{coarse-graining}to[out=0,in=180] (1.2,1.4);
		\draw[-latex,thick,red](-1.8,0.4)node[left]{coarse-graining}to[out=0,in=180] (3.5,-0.4);
		
	\end{tikzpicture}
	\caption{Coarse-graining for the one-dimensional Ising model.}
	\label{Fig:1}
\end{figure}
In principle, different coarse-graining procedures lead to different renormalization group schemes; see  \cite{Cardy1996book,Morandi2004document} for more detailed review of this procedure. Then a natural idea emerges, if we introduce a relationship which connects the original ``spins'' measure and new ``block spins'' measure from another point of view, it may bring some amazing observations with the help of new tools, which is exactly  what we attempt in this paper.

Different from the works before, from the perspective of dynamics, the Markov transition function could be constructed through a stochastic differential equation, more precisely, the stochastic quantization equation, which is proposed by Giorgio Parisi and Wu. Specificly, the basic idea of stochastic quantization is to consider the Euclidean path integral measure as the stationary distribution of a stochastic process; see \cite{Damgaard1987Huffel}. And according to this idea, Parisi and Wu formulated the concept of stochastic quantization as follows
\begin{enumerate}
	\item One supplements the fields $\phi(x)$ with an additional fictitious time $t$, which means $\phi(x)\rightarrow \phi(x,t)$;
	\item One demands that the fictitious time evolution of $\phi$ is described by the Langevin equation(a stochastic differential equation)
	\begin{align*}
		\frac{\partial \phi(x,t)}{\partial t} =-\frac{\delta S_{E} }{\delta \phi(x,t)}+\xi(x,t),
	\end{align*}
	where $\xi(x,t)$ is the time-spatial white noise, and $\delta S_{E}(\phi) /\delta \phi(x,t)$ represents the variation of effective action $S_{E}(\phi)$ with regard to field $\phi$. 
\end{enumerate}

Inspired by the idea of stochastic quantization, we begin to construct the transition probability function in renormalization  procedure of one-dimensional exactly solved Ising model through a Langevin equation.
From now on, we consider the Ising-like Hamiltonian without external field 
\begin{align*}
	S_{E}=\mathcal{H} =-K\sum_{i} \phi(x_{i})\phi(x_{i+1})+\frac{\gamma}{2}\sum_{i} \phi^{2}(x_{i}),
\end{align*}
where $\gamma$ is constant offset such that the Hamiltonian is positive definite; see \cite{Li2018document} for the similar reason when decoupling the Ising spins.

Then, calculating the variation of Hamiltonian $\mathcal{H}(\phi)$
\begin{align*}
	-\frac{\delta \mathcal{H} }{\delta \phi(x,t)} =K\sum_{i}\frac{\delta (\phi(x_{i},t)\phi(x_{i+1},t)) }{\delta \phi(x_{i})}-\gamma \phi(x_{i},t)= K[\phi(x_{i-1},t)+\phi(x_{i+1},t)]-\gamma \phi(x_{i},t),
\end{align*} 
we can get the stochastic quantization equation formally
\begin{align}\label{1.2}
	\frac{\partial \phi(x_{i},t)}{\partial t} &=K[ \phi(x_{i-1},t)+\phi(x_{i+1},t)-2\phi(x_{i},t)]+(2K-\gamma)\phi(x_{i},t)+\xi(x_{i},t)\nonumber\\
	&=K\tilde{\Delta}\phi(x_{i},t)+(2K-\gamma)\phi(x_{i},t)+\xi(x_{i},t),
\end{align} 
where $\xi(x_{i},t)$ is the time-spatial white noise, and $\tilde{\Delta}$ is the Laplacian on lattice
$$\tilde{\Delta}\phi(x_{i},t)= \phi(x_{i-1},t)+\phi(x_{i+1},t)-2\phi(x_{i},t).$$
If we use  Euler’s method in time discretisation\cite{Gyongydocument1999}, then we obtain the following explicit scheme
\begin{align}\label{eq:1.6}
	\phi(x_{i},t+1)=K\tilde{\Delta}\phi(x_{i},t)+(2K-\gamma+1)\phi(x_{i},t)+\xi(x_{i},t),
\end{align}

The intuition here is very clear, the informations of spins $\phi(x_{i},t+1)$ at time $t+1$ are totally determined by values of spins at time $t$ and spatial coordinates nearest $x_{i}$, in another word, $\phi(x_{i-1},t)$, $\phi(x_{i},t)$ and $\phi(x_{i+1},t)$, which is similar to the renormalization  procedure described above. The time $t$ above does not denote physical time, but rather an $RG \ time$, just as what Carosso did in paper \cite{Carosso2020document}. Obviously, the equation ~\eqref{eq:1.6}~can induce a Markov transition probability which connects the measure of ``orginal spins" at time $t$ and ``new block spins" at time $t+1$. Everything seems fascinating except that the spins $\phi(x_{i},t)$ here do not take values in $\{+1,-1\}$ owing to the noise we introduced. Fortunately, let's recall the renormalization procedure above carefully, the similar circumstance has ever occurred in the  coarse-graining for the one-dimensional Ising model. Physicists use~\eqref{1.4}~to give us a clever method to deal with this trouble
\begin{align*}
	S_{I}=sign\{\sigma^{I}_{1}+\sigma^{I}_{2}+\sigma^{I}_{3}\}.
\end{align*}
This is exactly what we need now. By this way, a new of renormalization  procedure can be established 
\begin{align}\label{eq:1.7}
	\phi(x_{i},t+1)=sgn\{K\tilde{\Delta}\phi(x_{i},t)+(2K-\gamma+1)\phi(x_{i},t)+\xi(x_{i},t)\},
\end{align}
Then the classical real space renormalization  procedure can be realized formally as above; see Figure \ref{Fig:2} for details. And the most important goal in this paper is to study the statistic mechanic properties of the discrete time stochastic dynamic above.
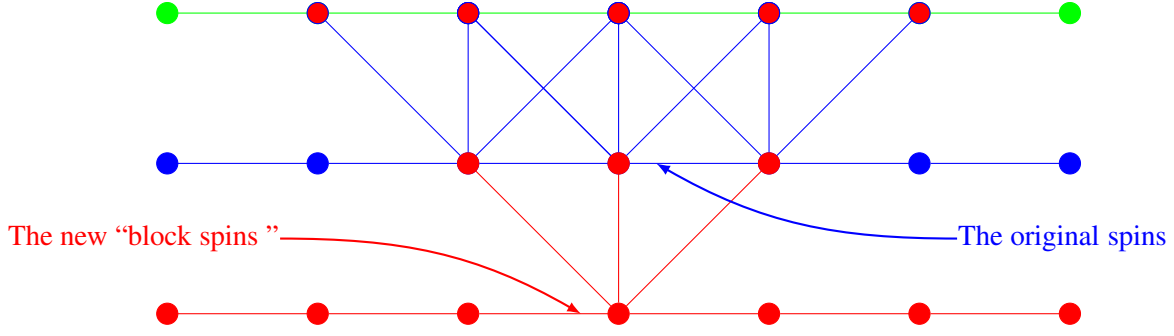
\begin{figure}
	\def\width{12}
	\def\hauteur{4}	
	\begin{tikzpicture}[x=1cm, y=1cm]

		\tikzstyle{every node}=[draw,circle,fill=green,minimum size=8pt,
		inner sep=0pt]            
		\draw [green](0,4) node (111){}
		-- ++(2,0) node (113)  {};
		\draw[green] (113){}
		-- ++(2,0) node (114)  {};
		\draw[green] (114){}
		-- ++(2,0) node (115)  {};
		\draw[green] (115){}
		-- ++(2,0) node (116)  {};
		\draw[green] (116){}
		-- ++(2,0) node (117)  {};
		\draw[green] (117){}
		-- ++(2,0) node (118)  {};

		\tikzstyle{every node}=[draw,circle,fill=blue,minimum size=8pt,
		inner sep=0pt]            
		\draw [blue](0,2) node (121){}
		-- ++(2,0) node (123)  {};
		\draw[blue] (123){}
		-- ++(2,0) node (124)  {};
		\draw[blue] (124){}
		-- ++(2,0) node (125)  {};
		\draw[blue] (125){}
		-- ++(2,0) node (126)  {};
		\draw[blue] (126){}
		-- ++(2,0) node (127)  {};
		\draw[blue] (127){}
		-- ++(2,0) node (128)  {};

		\tikzstyle{every node}=[draw,circle,fill=red,minimum size=8pt,
		inner sep=0pt]            
		\draw [red](0,0) node (131){}
		-- ++(2,0) node (133)  {};
		\draw[red] (133){}
		-- ++(2,0) node (134)  {};
		\draw[red] (134){}
		-- ++(2,0) node (135)  {};
		\draw[red] (135){}
		-- ++(2,0) node (136)  {};
		\draw[red] (136){}
		-- ++(2,0) node (137)  {};
		\draw[red] (137){}
		-- ++(2,0) node (138)  {};

		\draw[red] (135){}
		-- ++(-2,2) node (124)  {};
		\draw[red] (135){}
		-- ++(0,2) node (125)  {};
		\draw[blue] (124){}
		-- ++(-2,2) node (113)  {};
		\draw[blue] (124){}
		-- ++(0,2) node (114)  {};
		\draw[blue] (124){}
		-- ++(2,2) node (115)  {};
		\draw[blue] (125){}
		-- ++(-2,2) node (114)  {};
		\draw[blue] (125){}
		-- ++(-2,2) node (114)  {};
		\draw[blue] (125){}
		-- ++(0,2) node (115)  {};
		\draw[blue] (125){}
		-- ++(2,2) node (116)  {};

		\draw[red] (135){}
		-- ++(2,2) node (126)  {};
		\draw[blue] (126){}
		-- ++(2,2) node (117)  {};	
		\draw[blue] (126){}
		-- ++(0,2) node (116)  {};
		\draw[blue] (126){}
		-- ++(-2,2) node (115)  {};
		
		\tikzstyle{every node}=[minimum size=8pt,
		inner sep=0pt]

		\draw[-latex,thick,red](1.5,1)node[left]{The new ``block spins " }
		to[out=0,in=150] (5.5,0);	
		\draw[-latex,thick,blue](10.5,1)node[right]{The original spins}to[out=180,in=330] (6.5,2);
		
	\end{tikzpicture}
	\caption{The new ``block spins " are determined by the original spins.}
	\label{Fig:2}
\end{figure}
\section{The carr\'e du champ operator $\Gamma$ and Bakry-\'Emery curvature}\label{sec:3}
\subsection{The continuization of Markov chain }\label{sec:3.1}
First of all, let's recall some of the essential basic tools needed in the rest of
this paper. Consider the discrete time Markov chain $X_{n}$ defined on configuration space $E=\{+ 1,-1\}^{\mathbb{Z}}$, let $N_{t}$  be the Poission process with parameter $\lambda=1$, which is independent to $X_{n}$. Denote
\begin{align*}
	X_{t}=X_{N_{t}},
\end{align*}
which is a continuous time process with the generator
\begin{align}\label{eq:3.1}
	\mathcal{L}f(\eta )=\int_{\omega \in E} [f(\omega )-f(\eta)]p(\eta,d\omega ),
\end{align}
where $p(\eta,d\omega )$ represents the transition probability function of Markov chain $X_{n}$, and  
the semi-group corresponding to the generator $\mathcal{L}$ is given by	
\begin{align*}
	\mathcal{P}_{t}f(\eta )=\int_{\omega \in E} f(\omega )p(t,\eta,d\omega ).
\end{align*}
\begin{definition}
	Let $\mathcal{L}$ is the generator of semi-group $\mathcal{P}_{t}$, denote $\mathcal{D}(\mathcal{L})$ be the domain of the generator $\mathcal{L}$. Assume $\mathcal{A}$ is a dense subspace of $\mathcal{D}(\mathcal{L})$, which is algebra closed under multiplication. The carr\'e du champ operator $\Gamma$ can be defined as
	\begin{align*}
		2\Gamma(f,g)(\eta )=\mathcal{L}(fg)(\eta )-f(\eta )\mathcal{L}(g)(\eta )-g(\eta )\mathcal{L}(f)(\eta ),\ f,g \in \mathcal{A}.
	\end{align*}
	and $\Gamma_{2}$ can be defined by iterating $\Gamma$ as 
	\begin{align*}
		2\Gamma_{2}(f,g)(\eta )=\mathcal{L}(\Gamma(f,g))(\eta )-\Gamma(f,\mathcal{L}(g))(\eta )-\Gamma(g,\mathcal{L}(f))(\eta ),\ f,g \in \mathcal{A}.
	\end{align*}
	More explicitly, for the generator ~\eqref{eq:3.1}~, we have
	\begin{align*}
		2\Gamma(f,g)(\eta)&=\int_{\omega\in E}(f(\omega )-f(\eta))(g(\omega )-g(\eta))p(\eta,d\omega ),\\
		4\Gamma_{2}(f,f)(\eta)&=\int_{\omega\in E}\int_{z\in E}(f(z)-2f(\omega )+f(\eta))^{2}p(\omega,dz )p(\eta,d\omega )\nonumber\\
		&-2\{\int_{\omega\in E}f^{2}(\omega )p(\eta,d\omega )-\left(\int_{\omega\in E}f(\omega)p(\eta,d\omega )\right)^{2}\}.
	\end{align*}
\end{definition}
\subsection{The representation of $\Gamma$ and $\Gamma_{2}$}\label{sec:3.2}
From now on, we are going to consider a class of two-points functions $f(\eta)$, which only depend on spatial points $x_{1}$ and $x_{2}$, that is to say, for  $\eta \in E$, $f(\eta)=f(\eta(x_{1}),\eta(x_{2}))$. Conveniently, let
\begin{align*}
	E_{2}=\{(\eta(x_{1}),\eta(x_{2})):(+1,+1)=S_{1},
	(-1,+1)=S_{2},
	(+1,-1)=S_{3},
	(-1,-1)=S_{4}\},
\end{align*}
be the restriction on  coordinates $x_{1}$ and $x_{2}$  of configuration $\eta \in E$. Denote 
\begin{align}\label{eq:3.2.1}
	\mathbb{B}(\eta):= \left(\begin{array}{cccccccccc} 
		b_{1}(\eta)(T-t)\\
		b_{2}(\eta)(T-t)\\
		b_{3}(\eta)(T-t)\\
		b_{4}(\eta)(T-t)\end{array}\right )=\left(\begin{array}{cccccccccc} 
		p(T-t,\eta,S_{1})\\
		p(T-t,\eta,S_{2})\\
		p(T-t,\eta,S_{3})\\
		p(T-t,\eta,S_{4})\end{array}\right ),
	\vec{f}(\eta):= \left(\begin{array}{cccccccccc} 
		f(S_{1})\\
		f(S_{2})\\
		f(S_{3})\\
		f(S_{4})\end{array}\right ).
\end{align}
Then, we can express $F(t,\eta)=\mathcal{P}_{T-t}f(\eta )$ explicitly 
\begin{align}\label{eq:3.2}
	F(\eta):= F(t,\eta)=\mathcal{P}_{T-t}f(\eta )&=\int_{\omega\in E}f(\omega )p(T-t,\eta,d\omega )\nonumber\\
	&=\sum_{E_{2}}f(\omega(x_{1}),\omega(x_{2}) )p(T-t,\eta,(\omega(x_{1}),\omega(x_{2}) ))=
	(\vec{f})^{T}
	\cdot \mathbb{B}(\eta).
\end{align}
Due to equality~\eqref{eq:3.2}~above
\begin{align*}
	F(\omega )-F(\eta)=(\vec{f})^{T}
	\cdot \left(\mathbb{B}(\omega)- \mathbb{B}(\eta) \right ).
\end{align*}
Then operator $\Gamma$ can be written as
\begin{align}\label{eq:3.3}
	2\Gamma(F,F)(\eta)&=\int_{\omega\in E}(F(\omega )-F(\eta))^{2}p(\eta,d\omega )=(\vec{f})^{T}
	\cdot N(T-t,K)(\eta)
	\cdot \vec{f},
\end{align}
where
\begin{align*}
	N(T-t,K)(\eta)=\int_{\omega\in E}\left(\mathbb{B}(\omega)- \mathbb{B}(\eta) \right )
	\cdot \left(\mathbb{B}(\omega)- \mathbb{B}(\eta) \right )^{T}p(\eta,d\omega ).
\end{align*}
Noticed that for any $\omega \in E$,  $\sum_{k=1}^{4}b_{k}(\omega)=1$, then
\begin{align*}
	\mathbb{B}_{rem,1}(\omega):= Q_{1}\mathbb{B}(\omega)=\left(\begin{array}{cccccccccccccccc} 
		\sum_{k=1}^{4}b_{k}(\omega)\\
		b_{2}(\omega)\\
		b_{3}(\omega) \\
		b_{4}(\omega)
	\end{array}\right ) 
	=\left(\begin{array}{cccccccccccccccc} 
		1\\
		b_{2}(\omega)\\
		b_{3}(\omega) \\
		b_{4}(\omega)
	\end{array}\right ),\
	\text{where}\
	Q_{1}=\left(\begin{array}{cccccccccccccccc} 
		1&1&1&1\\
		0&1&0& 0\\
		0&0&1&0 \\
		0&0&0&1
	\end{array}\right ). 
\end{align*}
Hence, for operator $\Gamma$
\begin{align*}
	Q_{1}N(T-t,K)Q_{1}^{T}&=\int_{\omega\in E}Q_{1}\left(\mathbb{B}(\omega)- \mathbb{B}(\eta) \right)\cdot \left(\mathbb{B}(\omega)- \mathbb{B}(\eta) \right)^{T}Q_{1}^{T}p(\eta,d\omega )\\
	&=\int_{\omega\in E}\left(\mathbb{B}_{rem,1}(\omega)- \mathbb{B}_{rem,1}(\eta)\right)\cdot \left(\mathbb{B}_{rem,1}(\omega)- \mathbb{B}_{rem,1}(\eta)\right)^{T}p(\eta,d\omega )\\
	&=\left(\begin{array}{cccccccccc} 
		0&0&0&0\\
		0&\\
		0&&N_{rem,1}(T-t,K)\\
		0&\end{array}\right ).
\end{align*}
where $N_{rem,k}$ is the submatrix of $N$, which remove the $k$-th row and $k$-th column, and 
\begin{align*}
	(N_{rem,1})_{ij}=\int_{\omega\in E}(b_{i}(\omega)-b_{i}(\eta))(b_{j}(\omega)-b_{j}(\eta))p(\eta,d\omega ),i,j=2,3,4.
\end{align*}
Similarly, denote
\begin{align*}
	\mathbb{A}(z,\omega,\eta):= \left(\begin{array}{cccccccccc} 
		a_{1}(z,\omega,\eta)\\
		a_{2}(z,\omega,\eta)\\
		a_{3}(z,\omega,\eta)\\
		a_{4}(z,\omega,\eta)\end{array}\right )=\left(\begin{array}{cccccccccc} 
		b_{1}(z)-2b_{1}(\omega)+b_{1}(\eta)\\
		b_{2}(z)-2b_{2}(\omega)+b_{2}(\eta)\\
		b_{3}(z)-2b_{3}(\omega)+b_{3}(\eta)\\
		b_{4}(z)-2b_{4}(\omega)+b_{4}(\eta)\end{array}\right ),
\end{align*}
So
\begin{align*}
	F(z)-2F(\omega )+F(\eta)=&(\vec{f})^{T}
	\cdot \left(\mathbb{B}(z)-2\mathbb{B}(\omega)+ \mathbb{B}(\eta) \right )=(\vec{f})^{T} \cdot \mathbb{A}(z,\omega,\eta) .
\end{align*}
Then $\Gamma_{2}$ can be written as
\begin{align}\label{eq:3.4}
	4\Gamma_{2}(F,F)(\eta)&=\int_{\omega\in E}\int_{z\in E}(F(z)-2F(\omega )+F(\eta))^{2}p(\omega,dz )p(\eta,d\omega )\nonumber\\
	&-2\{\int_{\omega\in E}F^{2}(\omega )p(\eta,d\omega )-\left(\int_{\omega\in E}F(\omega)p(\eta,d\omega )\right)^{2}\}=(\vec{f})^{T}
	\cdot M(T-t,K)(\eta)
	\cdot \vec{f},
\end{align}
where
\begin{align*}
	&M(T-t,K)(\eta)=I^{1}(T-t,K)(\eta)-2I^{2}(T-t,K)(\eta)+2I^{3}(T-t,K)(\eta),\\
	&I^{1}(T-t,K)(\eta)=\int_{\omega\in E}\int_{z\in E}\mathbb{A}(z,\omega,\eta) \cdot \mathbb{A}^{T}(z,\omega,\eta)p(\omega,dz )p(\eta,d\omega ),\\
	&I^{2}(T-t,K)(\eta)=\int_{\omega\in E}\mathbb{B}(\omega) \cdot \mathbb{B}^{T}(\omega)p(\eta,d\omega ),\\
	&I^{3}(T-t,K)(\eta)=\int_{\omega\in E}\mathbb{B}(\omega)p(\eta,d\omega ) \cdot \int_{\omega\in E}\mathbb{B}^{T}(\omega)p(\eta,d\omega ).
\end{align*}
Hence, for operator $\Gamma_{2}$ 
\begin{align*}
	Q_{1}I^{1}(T-t,K)(\eta)Q_{1}^{T}&=\int_{\omega\in E}\int_{z\in E}Q_{1}\mathbb{A}(z,\omega,\eta) \cdot \mathbb{A}^{T}(z,\omega,\eta)Q_{1}^{T}p(\omega,dz )p(\eta,d\omega )\\
	&=\int_{\omega\in E}\int_{z\in E}\mathbb{A}_{rem,1}(z,\omega,\eta) \cdot \mathbb{A}_{rem,1}^{T}(z,\omega,\eta)p(\omega,dz )p(\eta,d\omega )\\
	&=\left(\begin{array}{cccccccccc} 
		0&0&0&0\\
		0&\\
		0&&A_{rem,1}\\
		0&\end{array}\right ),
\end{align*}
where $\mathbb{A}_{rem,1}(z,\omega,\eta)$ is similar to $\mathbb{B}_{rem,1}(\omega)$, ie $\mathbb{A}_{rem,1}(z,\omega,\eta):= Q_{1}\mathbb{A}(z,\omega,\eta)$, so
\begin{align*}
	(A_{rem,1})_{ij}=\int_{\omega\in E}\int_{z\in E}a_{i}(z,\omega,\eta)a_{j}(z,\omega,\eta)p(\omega,dz )p(\eta,d\omega ),i,j=2,3,4.
\end{align*}
Furthermore
\begin{align*}
	Q_{1}I^{2}(T-t,K)(\eta)Q_{1}^{T}&-Q_{1}I^{3}(T-t,K)(\eta)Q_{1}^{T}\\
	&=\int_{\omega\in E}Q_{1}\mathbb{B}(\omega) \cdot \mathbb{B}^{T}(\omega)Q_{1}^{T}p(\eta,d\omega )-\int_{\omega\in E}Q_{1}\mathbb{B}(\omega) p(\eta,d\omega ) \cdot \int_{\omega\in E}\mathbb{B}^{T}(\omega)Q_{1}^{T}p(\eta,d\omega )\\
	&=\int_{\omega\in E}\mathbb{B}_{rem,1}(\omega)  \mathbb{B}^{T}_{rem,1}(\omega)p(\eta,d\omega )-\int_{\omega\in E}\mathbb{B}_{rem,1}(\omega)p(\eta,d\omega ) \int_{\omega\in E}\mathbb{B}^{T}_{rem,1}(\omega)p(\eta,d\omega ).
\end{align*}
Then
\begin{align*}
	Q_{1}I^{2}(T-t,K)(\eta)Q_{1}^{T}-Q_{1}I^{3}(T-t,K)(\eta)Q_{1}^{T}=\left(\begin{array}{cccccccccc} 
		0&0&0&0\\
		0&\\
		0&&B_{rem,1}-C_{rem,1}\\
		0&\end{array}\right ),
\end{align*}
where
\begin{align*}
	(B_{rem,1})_{ij}=\int_{\omega\in E}b_{i}(\omega)b_{j}(\omega)p(\eta,d\omega ),\ (C_{rem,1})_{ij}=\int_{\omega\in E}b_{i}(\omega)p(\eta,d\omega )\int_{\omega\in E}b_{j}(\omega)p(\eta,d\omega ),i,j=2,3,4.
\end{align*}
In conclusion, $\Gamma$ and $\Gamma_{2}$ can be represented as
\begin{align}\label{eq:3.5}
	Q_{1}M(T-t,K)Q_{1}^{T}&=\left(\begin{array}{cccccccccc} 
		0&0&0&0\\
		0&\\
		0&&M_{rem,1}(T-t,K)\\
		0&\end{array}\right ),\nonumber\\
	Q_{1}N(T-t,K)Q_{1}^{T}
	&=\left(\begin{array}{cccccccccc} 
		0&0&0&0\\
		0&\\
		0&&N_{rem,1}(T-t,K)\\
		0&\end{array}\right ).
\end{align}
where $M_{rem,k}$ is the submatrix of $M$, which remove the $k$-th row and $k$-th column, and 
\begin{align*}
	(M_{rem,1})_{ij}=(A_{rem,1})_{ij}-2(B_{rem,1})_{ij}+2(C_{rem,1})_{ij},i,j=2,3,4.
\end{align*}
\subsection{The local Poincar\'e  inequality of two-points functions}\label{sec:3.3}
\begin{theorem}\label{thm:3.2}
	For two-points function $f$ defined as above, let  $F(t,\eta)=\mathcal{P}_{T-t}f(\eta ),t\in[0,T]$, $\mathcal{P}_{t}$ is the semi-group corresponding to a Markov process. If there exists a parameter such that the $4\times4$ matrices satisfy
	\begin{align}\label{eq:3.6}
		M(T-t,K)\geq \rho(T-t,K) N(T-t,K),
	\end{align}
	Then the  local Poincar\'e  inequality for any two-points function $f$ holds
	\begin{align}\label{eq:3.7}
		\mathcal{P}_{T}\left(f^{2}\right)(\eta)-(\mathcal{P}_{T}\left(f\right))^{2}(\eta) \leq 2\int_{0}^{T}e^{-\int_{t}^{T}\rho(T-s,K) ds}dt\cdot \mathcal{P}_{T}\left(\Gamma(f,f)\right)(\eta),
	\end{align}
	where $\eta\in E$ is the initial data. Furthermore, define 
	\begin{align*}
		\frac{1}{\rho}=\lim\limits_{T\rightarrow \infty}\int_{0}^{T}e^{-\int_{0}^{t}\rho(s,K) ds}dt.
	\end{align*}
	Assume the semi-group $\mathcal{P}_{t}$ is ergodic in the sense 
	that
	\begin{align*}
		\mu(f)=\lim\limits_{T\rightarrow \infty}(\mathcal{P}_{T}f)(\eta),
	\end{align*}
	where $\mu$ denote the probability measure and $\eta$ is an arbitrary initial data. Then the measure $\mu$ satisfies
	\begin{align}\label{eq:3.8}
		\mu\left(f^{2}\right)-(\mu\left(f\right))^{2}\leq \frac{2}{\rho} \mu\left(\Gamma(f,f)\right),
	\end{align}
	Particularly, if $\rho(t,K)\rightarrow \rho>0$ as $T\rightarrow +\infty$, the inequality~\eqref{eq:3.8}~holds. 
\end{theorem}
\begin{proof}
	Denote that $\Lambda(t,\eta )=\mathcal{P}_{t}\left(F^{2}(t,\cdot)\right)(\eta )=\left(\mathcal{P}_{t}\left(\mathcal{P}_{T-t} f\right)^{2}\right)(\eta )$, the first derivative of $\Lambda(t,\eta )$ with respect to time $t$ is
	\begin{align*}
		\Lambda^{'}(t,\eta )&=\frac{d}{dt}\left(\mathcal{P}_{t}\left(F^{2}(t,\cdot)\right)\right)(\eta )=\mathcal{P}_{t}\left(\mathcal{L}F^{2}(t,\cdot)\right)(\eta )+\mathcal{P}_{t}\left(\frac{d}{dt}F^{2}(t,\cdot)\right)(\eta )\\
		&=\mathcal{P}_{t}\left(\mathcal{L}F^{2}(t,\cdot)\right)(\eta )-2\mathcal{P}_{t}\left(F\mathcal{L}F(t,\cdot)\right)(\eta )=2\mathcal{P}_{t}\left(\Gamma(F,F)\right)(\eta),
	\end{align*}
	Since
	\begin{align*}
		2\frac{d}{dt}\Gamma(F,F)(\eta)&=\frac{d}{dt}\mathcal{L}F^{2}(t,\cdot)(\eta)+2\mathcal{L}F\mathcal{L}F(t,\cdot)(\eta)+2F\mathcal{L}\mathcal{L}F(t,\cdot)(\eta)\\
		&=-2\mathcal{L}\left(F\mathcal{L}F(t,\cdot)\right)(\eta)+2\mathcal{L}F\mathcal{L}F(t,\cdot)(\eta)+2F\mathcal{L}\mathcal{L}F(t,\cdot)(\eta)\\
		&=-4\Gamma(F(t,\cdot),\mathcal{L}F(t,\cdot))(\eta).
	\end{align*}
	Then the second derivative is given by
	\begin{align*}
		\Lambda^{''}(t,\eta )&=2\frac{d}{dt}\mathcal{P}_{t}\left(\Gamma(F,F)\right)(\eta)=2\mathcal{P}_{t}\left(\mathcal{L}\Gamma(F,F)\right)(\eta)+2\mathcal{P}_{t}\left(\frac{d}{dt}\Gamma(F,F)\right)(\eta )\\
		&=2\mathcal{P}_{t}\left(\mathcal{L}\Gamma(F,F)\right)(\eta)-4\mathcal{P}_{t}\left(\Gamma(F,\mathcal{L}F)\right)(\eta )=4\mathcal{P}_{t}\left(\Gamma_{2}(F,F)\right)(\eta).
	\end{align*}
	Owing to the expansion~\eqref{eq:3.3}~and~\eqref{eq:3.4}~, if condition~\eqref{eq:3.6}~satisfies, then 
	\begin{align*}
		\Gamma_{2}(F,F)(\eta)\geq \frac{\rho(T-t,K)}{2}\Gamma(F,F)(\eta),
	\end{align*} 
	which means
	\begin{align*}
		\Lambda^{''}(t,\eta )&\geq \rho(T-t,K) \Lambda^{'}(t,\eta ),
	\end{align*}
	Then
	\begin{align*}
		\Lambda^{'}(t,\eta )&\leq \Lambda^{'}(T,\eta )e^{-\int_{t}^{T}\rho(T-s,K)ds},
	\end{align*}
	Integrating both sides of last inequality from $t$ to $T$
	\begin{align*}
		\Lambda(T,\eta )-\Lambda(0,\eta )=\mathcal{P}_{T}\left(f^{2}\right)(\eta)-(\mathcal{P}_{T}\left(f\right))^{2}(\eta) \leq \Lambda^{'}(T,\eta )\int_{0}^{T}e^{-\int_{t}^{T}\rho(T-s,K) ds}dt.
	\end{align*}
	Noticed that $\Lambda^{'}(T,\eta )=2\mathcal{P}_{T}\left(\Gamma(f,f)\right)(\eta)$, then for any two-points function $f$, the following local Poincar\'e inequality holds
	\begin{align*}
		\mathcal{P}_{T}\left(f^{2}\right)(\eta)-(\mathcal{P}_{T}\left(f\right))^{2}(\eta) \leq 2\int_{0}^{T}e^{-\int_{t}^{T}\rho(T-s,K) ds}dt\cdot \mathcal{P}_{T}\left(\Gamma(f,f)\right)(\eta).
	\end{align*}
	According to the condition $\rho(t,K)\rightarrow \rho>0$, there exists $T^{*}$ large enough such that for $t\geq T^{*}$, $\rho(t,K)\geq \frac{\rho}{2}>0$, then
	\begin{align*}
		\mathcal{P}_{T}\left(f^{2}\right)(\eta)-(\mathcal{P}_{T}\left(f\right))^{2}(\eta) \leq& 2\int_{0}^{T}e^{-\int_{0}^{t}\rho(s,K) ds}dt\cdot \mathcal{P}_{T}\left(\Gamma(f,f)\right)(\eta)\\
		\leq&2\int_{0}^{T^{*}}e^{-\int_{0}^{t}\rho(s,K) ds}dt\cdot \mathcal{P}_{T}\left(\Gamma(f,f)\right)(\eta)\\
		+&2\int_{T^{*}}^{T}e^{-\int_{0}^{T^{*}}\rho(s,K) ds}e^{-\frac{\rho}{2}(t-T^{*})}dt\cdot \mathcal{P}_{T}\left(\Gamma(f,f)\right)(\eta)\\
		=&2\int_{0}^{T^{*}}e^{-\int_{0}^{t}\rho(s,K) ds}dt\cdot \mathcal{P}_{T}\left(\Gamma(f,f)\right)(\eta)\\
		+&2(\frac{2}{\rho}-\frac{2}{\rho}e^{\frac{-\rho(T-T^{*})}{2}})e^{-\int_{0}^{T^{*}}\rho(s,K) ds}\cdot \mathcal{P}_{T}\left(\Gamma(f,f)\right)(\eta).
	\end{align*}
	Let $T\rightarrow \infty$ firstly under the ergodic assumption, we can get
	\begin{align*}
		\mu\left(f^{2}\right)-(\mu\left(f\right))^{2} \leq&2(\int_{0}^{T^{*}}e^{-\int_{0}^{t}\rho(s,K) ds}dt+\frac{2}{\rho}e^{-\int_{0}^{T^{*}}\rho(s,K) ds})\cdot \mu\left(\Gamma(f,f)\right).
	\end{align*}
	Denote function  $h(t)=e^{-\int_{0}^{t}\rho(s,K) ds}$, because $\rho(t,K)\rightarrow \rho>0$, then $h(t)$ is decreasing on [$t_{0}$,+$\infty$), where $t_{0}$ is a constant sufficiently large.
	Recall that
	\begin{align*}
		\frac{1}{\rho}=\lim\limits_{T\rightarrow \infty}\int_{0}^{T}e^{-\int_{0}^{t}\rho(s,K) ds}dt=\lim\limits_{T\rightarrow \infty}\int_{0}^{T}h(t)dt,
	\end{align*}
	according to Lemma \ref{lem:A1} in appendix, we obtain
	$$\lim\limits_{t\rightarrow +\infty}th(t)=0.$$
	Since $T^{*}$ is large enough, then let $T^{*}\rightarrow \infty$
	\begin{align*}
		\lim\limits_{T^{*}\rightarrow \infty}e^{-\int_{0}^{T^{*}}\rho(s,K) ds}=0.
	\end{align*}
	Finally,
	\begin{align*}
		\mu\left(f^{2}\right)-(\mu\left(f\right))^{2} \leq&\frac{2}{\rho} \mu\left(\Gamma(f,f)\right),
	\end{align*}
	The proof of our main theorem is now complete.
\end{proof}
\begin{remark}
	A natural question is whether the Log-Sobolev inequality can be established by this way, unfortunately, for the generator defined on discrete configuration space, the Log-Sobolev inequality is very hard. It need some other methods to estimate the ralative entropy, which will be talked in our companion paper.
\end{remark}
\section{Application to the stochastic dynamic induced by real space renormalization procedure  }\label{sec:4}
\subsection{Our model and assumption}\label{sec:4.0}
For configuration  $\phi(x,t)=(\cdots,\phi(x_{-1},t),\phi(x_{0},t),\phi(x_{1},t),\cdots)$ in  $E=\{+ 1,-1\}^{\mathbb{Z}}$, let $\Pi_{\{i\}}$ be the projection of $\phi(x,t)$ at $i\in \mathbb{Z}$, which means $\Pi_{\{i\}}x=x_{i},\Pi_{\{i\}}\phi(x,t)=\phi(x_{i},t).$  Hence, equation~\eqref{eq:1.7}~in section \ref{sec:2} can be written as
\begin{align}\label{eq:1.8}
	\phi(x,t+1)=\mathcal{P}(\phi(x,t),\xi(x,t)),
\end{align}
where
\begin{align*}
	\Pi_{\{i\}}\mathcal{P}(\phi(x,t),\xi(x,t))=sgn\{K\tilde{\Delta}\phi(x_{i},t)+(2K-\gamma+1)\phi(x_{i},t)+\xi(x_{i},t)\}.
\end{align*}
The stochastic dynamic~\eqref{eq:1.8}~ can be seen as a Markov chain on state space $E=\{+ 1,-1\}^{\mathbb{Z}}$. Denote $X_{n}$ be the Markov chain induced by the stochastic dynamic~\eqref{eq:1.8}~, which satisfies
\begin{align}\label{eq:1.3}
	X(x_{i},n+1)=sgn\{K(X(x_{i}-1,n)+X(x_{i}+1,n))+(1-\gamma)X(x_{i},n)+\xi(x_{i},n)\},
\end{align}
where configuration is $\omega  \in E$, and $X_{n}$ is time-homogeneous when the model parameter $K$ is a constant. 

In this paper, we mainly establish the local Poincar\'e inequality through the  $\Gamma$ calculus, and get the estimate of two-points correlation function of the stochastic dynamic~\eqref{eq:1.8}~. Based on the extension of Bakry-\'Emery criterion in Section \ref{sec:3}, a series of matrix transforms are necessary in our proof, and
Figure \ref{Fig:3} provides a diagram sketching the relationship between the main lemmas.
\begin{figure}
	\begin{center}	
		\pgfdeclarelayer{background}
		\pgfdeclarelayer{foreground}
		\pgfsetlayers{background,main,foreground}
		
		
		\tikzstyle{sensor}=[draw, fill=blue!20, text width=3em, 
		text centered, minimum height=2.5em,rounded corners,drop shadow]
		\tikzstyle{Sensor}=[draw, fill=blue!20, text width=5em, 
		text centered, minimum height=2.5em,rounded corners,drop shadow]
		\tikzstyle{ann} = [above, text width=5em, text centered]
		\tikzstyle{wa} = [sensor, text width=5em, fill=orange!60, 
		minimum height=3em, rounded corners, drop shadow]
		\tikzstyle{sc} = [sensor, text width=10em, fill=red!50, 
		minimum height=3em, rounded corners, drop shadow]
		
		\def\blockdist{2.3}
		\def\edgedist{2.5}
		
		\begin{tikzpicture}
			\node (wa) [wa]  {Theorem 4.6};
			
			\path (wa.west)+(-1.5,0) node (asr5)[sensor] {Lemma 4.5};
			
			\path (asr5.north)+(0,1.5) node (asr4)[sensor] {Lemma 4.4}; 
			\path (asr4.east)+(1.5,0) node (asr2)[sensor] {Lemma 4.2};
			\path (asr2.north)+(0,1.5) node (asr3)[sensor] {Lemma 4.3};  
			\path (asr3.west)+(-1.5,0) node (asr1)[sensor] {Lemma 4.1};  
			\draw[fill=green!50][rounded corners = 0.2cm](1.5,-0.8)--(0.5,-0.8)[rounded corners = 0.2cm] 
			-- (0.5,-2.2)[sharp corners]--(1.5,-2.2)[sharp corners]--(2.5,-1.5)[sharp corners]--cycle;
			\path (1.4,-1.5) node (a) {Embedding};
			
			\draw[-latex,thick,blue](asr1)node[left]{}to[out=180,in=120] (asr5);
			\draw[-latex,thick,blue](asr3)node[left]{}to[out=270,in=90] (asr2);
			\draw[-latex,thick,blue](asr2)node[left]{}to[out=180,in=0] (asr4);
			\draw[-latex,thick,blue](asr4)node[left]{}to[out=270,in=90] (asr5);
			\draw[-latex,thick,blue](asr5)node[left]{}to[out=0,in=180] (wa);
			
			\path (wa.south) +(-1,-0.95*\blockdist) node (asrs) {Discrete time Markov chain};
			
			\begin{pgfonlayer}{background}
				\path (asr1.west |- asr1.north)+(-0.8,0.3) node (a) {};
				\path (wa.south -| wa.east)+(+0.3,-0.3) node (b) {};
				\path (wa.east |- asrs.east)+(+0.5,-0.6) node (c) {};
				
				\path[fill=yellow!20,rounded corners, draw=black!50, dashed]
				(a) rectangle (c);           
				\path (asr1.north west)+(-0.2,0.2) node (a) {};
				
			\end{pgfonlayer}
			
			\begin{pgfonlayer}{background}
				\path (asr3.west |- asr3.north)+(-0.5,0.3) node (a) {};
				\path (asr2.east -| asr2.north)+(0.5,-0.3) node (b) {};
				\path (asr2.south |- asr2.east)+(+1.5,-0.9) node (c) {};
				
				\path[fill=green!50,rounded corners, draw=black!50, dashed]
				(a) rectangle (c);           
				\path (asr2.east)+(0.7,0) node (d) {};
				
			\end{pgfonlayer}

			\path (wa.east)+(4,-1.5) node (syscomb) [sc] {Theorem 4.12};
			\path (syscomb.north)+(0,1) node (asrt10)[Sensor] {Proposition 4.10};
			\path (asrt10.north)+(0,1.5) node (asrt7)[Sensor] {Proposition 4.7};   
			\path (asrt7.west)+(-1,0) node (asrt8) [sensor] {Lemma 4.8};
			\path (asrt8.north)+(0,1.5) node (asrt9) [sensor] {Lemma 4.9};
			\path (asrt7.east)+(1,0) node (asrt11)[sensor] {Lemma 4.11};

			\draw[-latex,thick,blue](asrt8)node[left]{}to[out=270,in=90] (asrt10);			\draw[-latex,thick,blue](asrt9)node[left]{}to[out=270,in=90] (asrt8);
			\draw[-latex,thick,blue](asrt10)node[left]{}to[out=270,in=90] (syscomb);
			\draw[-latex,thick,blue](asrt7)node[left]{}to[out=270,in=90] (asrt10);
			\draw[-latex,thick,blue](asrt11)node[left]{}to[out=270,in=90] (asrt10);

			\draw[-latex,thick,green](wa)node[left]{}to[out=0,in=180] (asrt10);

			\path (syscomb.east)+(1.0,0.0) node (bwtn) {};
			
			\begin{pgfonlayer}{background}
				\path (asrt9.west |- asrt9.north)+(-0.5,0.3)node (g) {};
				\path (syscomb.east |- syscomb.south)+(1,-1.3) node (h) {};
				
				\path[fill=yellow!20,rounded corners, draw=black!50, dashed]
				(g) rectangle (h);

				\path (d.north)+(0,1.91) node (e) {};
				\draw[-latex,thick,green](e)node[below] {}to[out=0,in=180] (asrt9);
			\end{pgfonlayer}
			
			\path (syscomb.south) +(0,-0.3*\blockdist) 
			node (asrs) {Continuous time Markov process};
			
		\end{tikzpicture}
		\caption{The proof  sketch.}
		\label{Fig:3}	
	\end{center}
\end{figure}
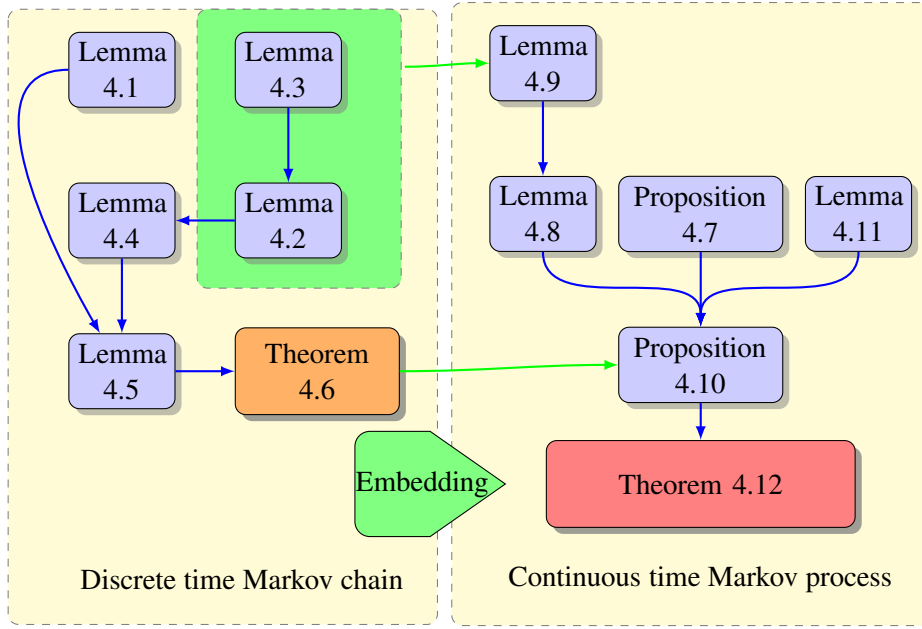
\subsection{The Bakry-\'Emery curvature of discrete time Markov chain}\label{sec:4.1}
In this section, the Markov chain $X(n)$ will be considered, let $p(n,\eta,d\omega)$ be the $n$ steps transition probability with initial data $\eta$, denote as above
$$F(n,\eta)=\mathcal{P}_{n}f(\eta )=\int f(\omega)p(n,\eta,d\omega).$$
As talked in section \ref{sec:3}, $\Gamma(F,F)(\eta)$ can be represented as
\begin{align}\label{eq:4.1}
	2\Gamma(F,F)(\eta)&=\int_{\omega\in E}(F(n,\omega )-F(n,\eta))^{2}p(\eta,d\omega )=(\vec{f})^{T}
	\cdot N(n,K)(\eta)
	\cdot \vec{f},
\end{align}
and $\Gamma_{2}(F,F)(\eta)$ can be represented as
\begin{align}\label{eq:4.2}
	4\Gamma_{2}(F,F)(\eta)&=\int_{\omega\in E}\int_{z\in E}(F(n,z)-2F(n,\omega )+F(n,\eta))^{2}p(\omega,dz )p(\eta,d\omega )\nonumber\\
	&-2\{\int_{\omega\in E}F^{2}(n,\omega )p(\eta,d\omega )-\left(\int_{\omega\in E}F(n,\omega)p(\eta,d\omega )\right)^{2}\}=(\vec{f})^{T}
	\cdot M(n,K)(\eta)
	\cdot \vec{f}.
\end{align}
Obviously, $N(n,K)(\eta)$ and $M(n,K)(\eta)$ share the same structure with $N(T-t,K)(\eta)$ and $M(T-t,K)(\eta)$ respectively, which means  both $N(n,K)(\eta)$ and $M(n,K)(\eta)$ satisfies \eqref{eq:3.5}~.

To calculate $N(n,K)(\eta)$ and $M(n,K)(\eta)$ more concretely, we consider a special case, where the system parameter $K=0$.  For ease of notations, introduce the following symbols
\begin{align*}
	\left(\begin{array}{cccccccccc} 
		a_{i,1}\\
		a_{i,2}\\
		a_{i,3}\\
		a_{i,4}\\
		a_{i,5}\\
		a_{i,6}\\
		a_{i,7}\\
		a_{i,8}\end{array}\right ) =\left(\begin{array}{cccccccccc} 
		a_{i}(S_{1},S_{1},\eta)\\
		a_{i}(S_{2},S_{1},\eta)\\
		a_{i}(S_{3},S_{1},\eta)\\
		a_{i}(S_{4},S_{1},\eta)\\
		a_{i}(S_{1},S_{2},\eta)\\
		a_{i}(S_{2},S_{2},\eta)\\
		a_{i}(S_{3},S_{2},\eta)\\
		a_{i}(S_{4},S_{2},\eta)\end{array}\right ),
	\left(\begin{array}{cccccccccc} 
		p(S_{1},S_{1})p(S_{1},S_{1})\\
		p(S_{1},S_{2})p(S_{1},S_{1})\\
		p(S_{1},S_{3})p(S_{1},S_{1})\\
		p(S_{1},S_{4})p(S_{1},S_{1})\\
		p(S_{2},S_{1})p(S_{1},S_{2})\\
		p(S_{2},S_{2})p(S_{1},S_{2})\\
		p(S_{2},S_{3})p(S_{1},S_{2})\\
		p(S_{2},S_{4})p(S_{1},S_{2})\\
	\end{array}\right )=
	\left(\begin{array}{cccccccccc} 
		p_{1}\\
		p_{2}\\
		p_{3}\\
		p_{4}\\
		p_{5}\\
		p_{6}\\
		p_{7}\\
		p_{8}\end{array}\right ) 
	=\left(\begin{array}{cccccccccc} 
		\Phi^{4}(1-\gamma)\\
		\Phi^{3}(1-\gamma)\Phi(-(1-\gamma))\\
		\Phi^{3}(1-\gamma)\Phi(-(1-\gamma))\\
		\Phi^{2}(1-\gamma)\Phi^{2}(-(1-\gamma))\\
		\Phi^{2}(1-\gamma)\Phi^{2}(-(1-\gamma))\\
		\Phi^{3}(1-\gamma)\Phi(-(1-\gamma))\\
		\Phi(1-\gamma)\Phi^{3}(-(1-\gamma))\\
		\Phi^{2}(1-\gamma)\Phi^{2}(-(1-\gamma))\\
	\end{array}\right ),
\end{align*}
and
\begin{align*}
	\left(\begin{array}{cccccccccc} 
		a_{i,9}\\
		a_{i,10}\\
		a_{i,11}\\
		a_{i,12}\\
		a_{i,13}\\
		a_{i,14}\\
		a_{i,15}\\
		a_{i,16}\end{array}\right ) =\left(\begin{array}{cccccccccc} 
		a_{i}(S_{1},S_{3},\eta)\\
		a_{i}(S_{2},S_{3},\eta)\\
		a_{i}(S_{3},S_{3},\eta)\\
		a_{i}(S_{4},S_{3},\eta)\\
		a_{i}(S_{1},S_{4},\eta)\\
		a_{i}(S_{2},S_{4},\eta)\\
		a_{i}(S_{3},S_{4},\eta)\\
		a_{i}(S_{4},S_{4},\eta)\end{array}\right ),
	\left(\begin{array}{cccccccccc} 
		p(S_{3},S_{1})p(S_{1},S_{3})\\
		p(S_{3},S_{2})p(S_{1},S_{3})\\
		p(S_{3},S_{3})p(S_{1},S_{3})\\
		p(S_{3},S_{4})p(S_{1},S_{3})\\
		p(S_{4},S_{1})p(S_{1},S_{4})\\
		p(S_{4},S_{2})p(S_{1},S_{4})\\
		p(S_{4},S_{3})p(S_{1},S_{4})\\
		p(S_{4},S_{4})p(S_{1},S_{4})\end{array}\right )=
	\left(\begin{array}{cccccccccc} 
		p_{9}\\
		p_{10}\\
		p_{11}\\
		p_{12}\\
		p_{13}\\
		p_{14}\\
		p_{15}\\
		p_{16}\end{array}\right ) 
	=\left(\begin{array}{cccccccccc} 
		\Phi^{2}(1-\gamma)\Phi^{2}(-(1-\gamma))\\
		\Phi(1-\gamma)\Phi^{3}(-(1-\gamma))\\
		\Phi^{3}(1-\gamma)\Phi(-(1-\gamma))\\
		\Phi^{2}(1-\gamma)\Phi^{2}(-(1-\gamma))\\
		\Phi^{4}(-(1-\gamma))\\
		\Phi(1-\gamma)\Phi^{3}(-(1-\gamma))\\
		\Phi(1-\gamma)\Phi^{3}(-(1-\gamma))\\
		\Phi^{2}(1-\gamma)\Phi^{2}(-(1-\gamma))\end{array}\right ),
\end{align*}
Recall that the definition in~\eqref{eq:3.2.1}~, we define 
\begin{align}\label{eq:4.3}
	\left(\begin{array}{cccccccccc} 
		b_{i,1}(n)\\
		b_{i,2}(n)\\
		b_{i,3}(n)\\
		b_{i,4}(n)\end{array}\right ) =\left(\begin{array}{cccccccccc} 
		b_{i}(S_{1})(n)\\
		b_{i}(S_{2})(n)\\
		b_{i}(S_{3})(n)\\
		b_{i}(S_{4})(n)\end{array}\right ),
	\tilde{p}=\left(\begin{array}{cccccccccc} 
		\tilde{p}_{1}\\
		\tilde{p}_{2}\\
		\tilde{p}_{3}\\
		\tilde{p}_{4}\end{array}\right )
	=\left(\begin{array}{cccccccccc} 
		\Phi^{2}(1-\gamma)\\
		\Phi(1-\gamma)\Phi(-(1-\gamma))\\
		\Phi(1-\gamma)\Phi(-(1-\gamma))\\
		\Phi^{2}(-(1-\gamma))\end{array}\right ).
\end{align}
\begin{lemma}\label{lem:4.1} 
	Based on the symbols introduced above, define
	\begin{align*}
		P_{1}&=diag\{p(S_{1},S_{1}),p(S_{1},S_{2}),p(S_{1},S_{3}),p(S_{1},S_{4})\},\\	P_{2}&=diag\{p(S_{2},S_{1}),p(S_{2},S_{2}),p(S_{2},S_{3}),p(S_{2},S_{4})\},\\
		P_{3}&=diag\{p(S_{3},S_{1}),p(S_{3},S_{2}),p(S_{3},S_{3}),p(S_{3},S_{4})\},\\	
		P_{4}&=diag\{p(S_{4},S_{1}),p(S_{4},S_{2}),p(S_{4},S_{3}),p(S_{4},S_{4})\}.
	\end{align*}
	then
	\begin{align*}
		P_{i}&=P_{1i}P_{jk}P_{1}P_{jk}P_{1i}\\
		&=P_{1i}P_{jk}\cdot diag\{p(S_{1},S_{1}),p(S_{1},S_{2}),p(S_{1},S_{3}),p(S_{1},S_{4})\} \cdot P_{jk}P_{1i},i,j,k=2,3,4,i\neq j\neq k.
	\end{align*}
	where $P_{ij}$ is the matrix which exchanges the $i$-th row and $j$-th column of identity matrix. Furthermore,  $diag\{p_{1},p_{2},\cdots,p_{16}\}=diag\{\tilde{p}_{1}P_{1},\tilde{p}_{2}P_{2},\tilde{p}_{3}P_{3},\tilde{p}_{4}P_{4}\}:= diag\{\tilde{P_{1}},\tilde{P_{2}},\tilde{P_{3}},\tilde{P_{4}}\}.$
\end{lemma}
\begin{proof}
	This conclusion can be checked straightforward, the proof will be omitted.
\end{proof}
\begin{lemma}\label{lem:4.2}
	According to~\eqref{eq:4.3}~, define
	\begin{align*}
		B(n):=\left(\begin{array}{cccccccccc} 
			b_{1,1}(n) & b_{1,2}(n)& b_{1,3}(n) & b_{1,4}(n)\\
			b_{2,1}(n)& b_{2,2}(n)& b_{2,3}(n) & b_{2,4}(n)\\
			b_{3,1}(n) & b_{3,2}(n)& b_{3,3}(n)& b_{3,4}(n)\\
			b_{4,1}(n) &b_{4,2}(n)& b_{4,3}(n) & b_{4,4}(n)\end{array}\right ).
	\end{align*}
	Then for any positive integer $n$, B(n) are always invertible.
\end{lemma}	
\begin{proof}
	In the case of $K=0$, noticed that the Gaussian distribution at point $x_{1}$ and $x_{2}$ are independent, then for $s_{i}\in \{\pm 1\},i=1,2,3,4$, 
	\begin{align*}
		p(n,(x_{1}=s_{1},x_{2}=s_{2}),(x_{1}=s_{3},x_{2}=s_{4}))=&p(n,x_{1}=s_{1},x_{1}=s_{3})\cdot p(n,x_{2}=s_{2},x_{2}=s_{4}),
	\end{align*}
	and
	$$p(n,x_{1}=s_{i},x_{1}=s_{j})= p(n,x_{2}=s_{i},x_{2}=s_{j}).$$
	Denote
	\begin{align}\label{eq:5.6}
		\left\{\begin{array}{cccccccccc}
			y_{1}(n)=p(n,x_{1}=+1,x_{1}=+1)= p(n,x_{2}=+1,x_{2}=+1),\\
			y_{2}(n)=p(n,x_{1}=-1,x_{1}=+1)= p(n,x_{2}=-1,x_{2}=+1),\\
			y_{3}(n)=p(n,x_{1}=+1,x_{1}=-1)= p(n,x_{2}=+1,x_{2}=-1),\\
			y_{4}(n)=p(n,x_{1}=-1,x_{1}=-1)= p(n,x_{2}=-1,x_{2}=-1).\end{array}\right.
	\end{align}
	Obviously, $y_{1}+y_{3}=1$ and $y_{2}+y_{4}=1$. According to Lemma \ref{lem:4.3} below, we know that $y_{1}=y_{4}$ and $y_{2}=y_{3}$, there is a constant $c=c(n)>0$ but $c(n)\neq1$ such that  $y_{2}=c(n)\cdot y_{1}$. Hence
	\begin{align*}
		B(n)
		=\left(\begin{array}{cccccccccc} 
			y^{2}_{1}& y_{1}y_{2}& y_{1}y_{2} & y^{2}_{2}\\
			y_{1}y_{3} & y_{1}y_{4}& y_{2}y_{3} & y_{2}y_{4}\\
			y_{1}y_{3} & y_{2}y_{3}& y_{1}y_{4}& y_{2}y_{4}\\
			y^{2}_{3} &y_{3}y_{4}& y_{3}y_{4} &y^{2}_{4}\end{array}\right )
		=y^{2}_{1}\left(\begin{array}{cccccccccc} 
			1& c& c & c^{2}\\
			c& 1&c^{2} & c\\
			c & c^{2}& 1& c\\
			c^{2} &c& c &1\end{array}\right).
	\end{align*} 
	It is easy to check that $B(n)$ is invertible if $c(n)\neq1$, which completes the proof of Lemma \ref{lem:4.2}. 
\end{proof}
\begin{lemma}\label{lem:4.3}
	Assume that $y_{1}(n),y_{2}(n),y_{3}(n),y_{4}(n)$ are the values introduced in Lemma \ref{lem:4.2}, then
	\begin{align*}
		\left\{\begin{array}{cccccccccc} 
			y_{1}(n)=(\Phi(1-\gamma)-\frac{1}{2})(\Phi(1-\gamma)-\Phi(\gamma-1))^{n-1}+\frac{1}{2},\\
			y_{2}(n)=\frac{1}{2}-(\Phi(1-\gamma)-\frac{1}{2})(\Phi(1-\gamma)-\Phi(\gamma-1))^{n-1},\\
			y_{3}(n)=\frac{1}{2}-(\Phi(1-\gamma)-\frac{1}{2})(\Phi(1-\gamma)-\Phi(\gamma-1))^{n-1},\\
			y_{4}(n)=\frac{1}{2}+(\frac{1}{2}-\Phi(\gamma-1))(\Phi(1-\gamma)-\Phi(\gamma-1))^{n-1}.\end{array}\right.
	\end{align*}
\end{lemma}	
\begin{proof}
	For any positive integer $n$, denote that $E_{n}=p(n,x_{1}=+1,x_{1}=+1)$ and $F_{n}=p(n,x_{1}=+1,x_{1}=-1)$. Then $E_{n}+F_{n}=1$, and satisfies the following equation
	\begin{align*}
		E_{n}&=p(n,x_{1}=+1,x_{1}=+1)=p(n-1,x_{1}=+1,x_{1}=+1)p(1,x_{1}=+1,x_{1}=+1)\\
		&+p(n-1,x_{1}=+1,x_{1}=-1)p(1,x_{1}=-1,x_{1}=+1)\\
		&=E_{n-1}\Phi(1-\gamma)+F_{n-1}\Phi(\gamma-1)=E_{n-1}\Phi(1-\gamma)+(1-E_{n-1})\Phi(\gamma-1)\\
		&=E_{n-1}(\Phi(1-\gamma)-\Phi(\gamma-1))+\Phi(\gamma-1).
	\end{align*}
	we know that 
	\begin{align*}
		E_{n}&=E_{1}(\Phi(1-\gamma)-\Phi(\gamma-1))^{n-1}+\frac{1-(\Phi(1-\gamma)-\Phi(\gamma-1))^{n-1}}{2}.
	\end{align*}
	Noticed that  $E_{1}=p(1,x_{1}=+1,x_{1}=+1)=\Phi(1-\gamma)$, then
	\begin{align*}
		E_{n}&=(\Phi(1-\gamma)-\frac{1}{2})(\Phi(1-\gamma)-\Phi(\gamma-1))^{n-1}+\frac{1}{2}.
	\end{align*}
	Similarly, if $E_{n}=p(n,x_{1}=-1,x_{1}=+1)$ and $F_{n}=p(n,x_{1}=-1,x_{1}=-1)$, there admits the equation  
	\begin{align*}
		E_{n}&=p(n,x_{1}=-1,x_{1}=+1)=p(n-1,x_{1}=-1,x_{1}=+1)p(1,x_{1}=+1,x_{1}=+1)\\
		&+p(n-1,x_{1}=-1,x_{1}=-1)p(1,x_{1}=-1,x_{1}=+1)\\
		&=E_{n-1}\Phi(1-\gamma)+F_{n-1}\Phi(\gamma-1)=E_{n-1}\Phi(1-\gamma)+(1-E_{n-1})\Phi(\gamma-1)\\
		&=E_{n-1}(\Phi(1-\gamma)-\Phi(\gamma-1))+\Phi(\gamma-1),
	\end{align*}
	Hence
	\begin{align*}
		E_{n}&=E_{1}(\Phi(1-\gamma)-\Phi(\gamma-1))^{n-1}+\frac{1-(\Phi(1-\gamma)-\Phi(\gamma-1))^{n-1}}{2}.
	\end{align*}
	Owing to $E_{1}=p(1,x_{1}=-1,x_{1}=+1)=\Phi(\gamma-1)$, then
	\begin{align*}
		E_{n}&=\frac{1}{2}-(\frac{1}{2}-\Phi(\gamma-1))(\Phi(1-\gamma)-\Phi(\gamma-1))^{n-1}\\
		&=\frac{1}{2}-(\Phi(1-\gamma)-\frac{1}{2})(\Phi(1-\gamma)-\Phi(\gamma-1))^{n-1}.
	\end{align*}
	The proof is completed.
\end{proof}
Next, we can define the most important matrices in our calculation for general system parameter $K$. 
\begin{lemma}\label{lem:4.4}
	Consider the system with arbitrary parameter $K$ and initial data $\eta$, for any positive integer $n$, define
	\begin{align}\label{eq:4.5}
		M_{1}^{*}(n,K)(\eta)&=Q_{1}B^{-1}(n)M(n,K)(\eta)(Q_{1}B^{-1}(n))^{T},\nonumber\\
		N_{1}^{*}(n,K)(\eta)&=Q_{1}B^{-1}(n)N(n,K)(\eta)(Q_{1}B^{-1}(n))^{T}.
	\end{align}
	In case that does not give rise to misunderstandings, the initial data $\eta$ will be omitted usually. Then we can get
	\begin{align*}
		M_{1}^{*}(n,K)&=\left(\begin{array}{cccccccccc} 
			0&0&0&0\\
			0&\\
			0&&M_{rem,1}^{*}(n,K)\\
			0&\end{array}\right ),\\
		N_{1}^{*}(n,K)&=\left(\begin{array}{cccccccccc} 
			0&0&0&0\\
			0&\\
			0&&N_{rem,1}^{*}(n,K)\\
			0&\end{array}\right ).
	\end{align*}
	Furthermore, define 
	\begin{align*}
		M_{i}^{*}(n,K)(\eta)&=Q_{i}B^{-1}(n)M(n,K)(\eta)(Q_{i}B^{-1}(n))^{T},\\
		N_{i}^{*}(n,K)(\eta)&=Q_{i}B^{-1}(n)N(n,K)(\eta)(Q_{i}B^{-1}(n))^{T},i=1,2,3,4.
	\end{align*}
	Then for any initial data $\eta$, the elements in the $i$-th row and $i$-th column of  $M_{i}^{*}(n,K)$ and $N_{i}^{*}(n,K)$ are zeros, where
	\begin{align*}
		Q_{2}=\left(\begin{array}{cccccccccccccccc} 
			1&0&0& 0\\
			1&1&1&1\\
			0&0&1&0 \\
			0&0&0&1
		\end{array}\right ),	
		Q_{3}=\left(\begin{array}{cccccccccccccccc}  
			1&0&0& 0\\
			0&1&0&0 \\
			1&1&1&1\\
			0&0&0&1
		\end{array}\right ), 
		Q_{4}=\left(\begin{array}{cccccccccccccccc} 
			1&0&0&0\\
			0&1&0& 0\\
			0&0&1&0 \\
			1&1&1&1
		\end{array}\right ) .
	\end{align*}
\end{lemma}	
\begin{proof}
	According to equation ~\eqref{eq:3.5}~ in section  \ref{sec:3.2}, for any $n\in \mathbb{N}^{+}$ and parameter $K$ 	
	\begin{align*}
		Q_{1}M(n,K)Q_{1}^{T}&=\left(\begin{array}{cccccccccc} 
			0&0&0&0\\
			0&\\
			0&&M_{rem,1}(n,K)\\
			0&\end{array}\right ),\\
		Q_{1}N(n,K)Q_{1}^{T}
		&=\left(\begin{array}{cccccccccc} 
			0&0&0&0\\
			0&\\
			0&&N_{rem,1}(n,K)\\
			0&\end{array}\right ).
	\end{align*}
	Due to Lemma \ref{lem:4.2}
	\begin{align*}
		B(n)=y^{2}_{1}(n)\left(\begin{array}{cccccccccc} 
			1& c(n)& c(n) & c^{2}(n)\\
			c(n)& 1&c^{2}(n) & c(n)\\
			c(n) & c^{2}(n)& 1& c(n)\\
			c^{2}(n) &c(n)& c(n) &1\end{array}\right ).
	\end{align*} 
	Hence
	\begin{align*}
		Q_{1}B(n)Q_{1}^{-1}=y^{2}_{1}(n)\left(\begin{array}{cccccccccc} 
			(1+c(n))^{2}&0&0&0\\
			c(n)& 1-c(n)&c^{2}(n)-c(n) & 0\\
			c(n) & c^{2}(n)-c(n)& 1-c(n)& 0\\
			c^{2}(n) &c(n)-c^{2}(n)& c(n)-c^{2}(n) &1-c^{2}(n)\end{array}\right ).
	\end{align*}
	So
	\begin{align}\label{eq:4.6}
		Q_{1}B(n)^{-1}Q_{1}^{-1}&=\frac{1}{y^{2}_{1}(n)}\left(\begin{array}{cccccccccc} 
			(1+c(n))^{2}&0&0&0\\
			c(n)& 1-c(n)&c^{2}(n)-c(n) & 0\\
			c(n) & c^{2}(n)-c(n)& 1-c(n)& 0\\
			c^{2}(n) &c(n)-c^{2}(n)& c(n)-c^{2}(n) &1-c^{2}(n)\end{array}\right )^{-1}\nonumber\\
		&=\frac{1}{y^{2}_{1}(n)}\left(\begin{array}{cccccccccc} 
			\frac{1}{(1+c(n))^{2}}&0&0&0\\
			\frac{-c(n)}{(1+c(n))^{2}(1-c(n))^{2}}& \frac{1}{(1+c(n))(1-c(n))^{2}}&\frac{c(n)}{(1+c(n))(1-c(n))^{2}} & 0\\
			\frac{-c(n)}{(1+c(n))^{2}(1-c(n))^{2}} & \frac{c(n)}{(1+c(n))(1-c(n))^{2}}& \frac{1}{(1+c(n))(1-c(n))^{2}}& 0\\
			\frac{c^{2}(n)}{(1+c(n))^{2}(1-c(n))^{2}} &\frac{-c(n)}{(1+c(n))(1-c(n))^{2}}&\frac{-c(n)}{(1+c(n))(1-c(n))^{2}} &\frac{1}{1-c^{2}(n)}\end{array}\right ).
	\end{align}
	With some long but straightforward algebra, we find 
	\begin{align*}
		M_{1}^{*}(n,K)&=Q_{1}B(n)^{-1}M(n,K)(Q_{1}B(n)^{-1})^{T}\\
		&=Q_{1}B(n)^{-1}Q_{1}^{-1}Q_{1}M(n,K)Q_{1}^{T}(Q_{1}B(n)^{-1}Q_{1}^{-1})^{T}=\left(\begin{array}{cccccccccc} 
			0&0&0&0\\
			0&\\
			0&&M_{rem,1}^{*}(n,K)\\
			0&\end{array}\right ),\\
		N_{1}^{*}(n,K)&=Q_{1}B(n)^{-1}N(n,K)(Q_{1}G^{-1})^{T}\\
		&=Q_{1}B(n)^{-1}Q_{1}^{-1}Q_{1}N(n,K)Q_{1}^{T}(Q_{1}B(n)^{-1}Q_{1}^{-1})^{T}=\left(\begin{array}{cccccccccc} 
			0&0&0&0\\
			0&\\
			0&&N_{rem,1}^{*}(n,K)\\
			0&\end{array}\right ).
	\end{align*}
	Furthermore 
	\begin{align*}
		M_{i}^{*}(n,K)&=Q_{i}B(n)^{-1}M(n,K)(Q_{i}B(n)^{-1})^{T}\\
		&=Q_{i}Q^{-1}_{1}Q_{1}B(n)^{-1}M(n,K)(B(n)^{-1})^{T}Q^{T}_{1}(Q^{-1}_{1})^{T}Q_{i}^{T}\\
		&=Q_{i}Q^{-1}_{1}\left(\begin{array}{cccccccccc}
			0&0&0&0\\
			0&\\
			0&&M_{rem,1}^{*}(n,K)\\
			0&\end{array}\right )(Q^{-1}_{1})^{T}Q_{i}^{T}.
	\end{align*}
	Noticed that
	\begin{align*}
		Q_{2}Q^{-1}_{1}=\left(\begin{array}{cccccccccc} 
			1&-1&-1&-1\\
			1&0&0&0\\
			0&0&1&0\\
			0&0&0&1\end{array}\right ),
	\end{align*}
	then the second row and the second column of matrix $M_{2}^{*}(n,K)$ are 0. It is easy to see, the same results can be checked for the cases $i=3,4.$ 
\end{proof} 
In the rest of this subsection, we will focus on the Bakry-\'Emery curvature in the case $K=0$. Without loss of generality, choosing the initial data to be $\eta=S_{1}$ firstly. If it does not cause ambiguity, for square field operator $\Gamma$, the $4\times 4$ matrix $N(n,K=0)(\eta=S_{1})$ which determines the expression of $\Gamma(F,F)(\eta=S_{1})$ will be recorded as $N(n,K=0)$. The time variable $T-t$ used in equation~\eqref{eq:3.2.1}~ is substituted by $n$ and will be omitted if it is not necessary. The $(i,j)$ element of $N(n,K=0)$ is
\begin{align*}
	N(n,K=0)_{ij}&=\int_{\omega\in \{S_{1},S_{2},S_{3},S_{4}\}}
	(b_{i}(\omega)-b_{i}(S_{1}))
	(b_{j}(\omega)-b_{j}(S_{1})))p(S_{1},\omega )\\
	&=\sum_{k=1}^{4}(b_{i,k}-b_{i,1})(b_{j,k}-b_{j,1})\tilde{p}_{k},
\end{align*}
To be convenient, let
\begin{align*}
	D_{1}=
	\left(\begin{array}{cccccccccc} 
		0 & -1& -1 &-1\\
		0 & 1& 0 &0\\
		0 & 0& 1 & 0\\
		0 & 0& 0 & 1\end{array}\right ),
	diag\{p_{1},\cdots,p_{k}\}=\left(\begin{array}{cccccccccc} 
		p_{1}&&\\
		&\ddots \\
		&&p_{k}\end{array}\right ), \ for\ integer \ k\geq 1.
\end{align*}
Hence, in the matrix form, $N(n,K=0)$ can be written as 
\begin{align*}
	N(n,K=0)=BD_{1}diag\{\tilde{p}_{1},\tilde{p}_{2},\tilde{p}_{3},\tilde{p}_{4}\} D_{1}^{T} B^{T}=B\hat{P}_{1} B^{T},
\end{align*}
where
\begin{align*}
	\hat{P}_{1}=\left(\begin{array}{cccccccccc} 
		\tilde{p}_{2}+\tilde{p}_{3}+\tilde{p}_{4}&-\tilde{p}_{2}&-\tilde{p}_{3}&-\tilde{p}_{4}\\
		-\tilde{p}_{2}&\tilde{p}_{2}&0&0\\
		-\tilde{p}_{3}&0&\tilde{p}_{3}&0\\
		-\tilde{p}_{4}&0&0&\tilde{p}_{4}\end{array}\right ).
\end{align*}
So
\begin{align*}
	N_{1}^{*}(n,K=0)=Q_{1}B^{-1}N(n,K=0)(B^{T})^{-1} Q_{1}^{T}=diag\{0,\tilde{p}_{2},\tilde{p}_{3},\tilde{p}_{4}\}.
\end{align*}
However, the operator $\Gamma_{2}(F,F)(\eta=S_{1})$ is determined by matrix $$M(n,K=0)=I^{1}(n,K=0)(S_{1})-2I^{2}(n,K=0)(S_{1})+2I^{3}(n,K=0)(S_{1}),$$
where  $I^{1}(n,K=0)(S_{1}),\ I^{2}(n,K=0)(S_{1}),\ I^{3}(n,K=0)(S_{1})$ are all $4\times 4$ matrices with elements
\begin{align*}
	I^{1}_{ij}&=\int_{\omega\in \{S_{1},S_{2},S_{3},S_{4}\}}\int_{z\in  \{S_{1},S_{2},S_{3},S_{4}\}}
	a_{i}(z,\omega,\eta )\cdot 
	a_{j}(z,\omega,\eta )p(\omega,z )p(\eta,\omega )=\sum_{k=1}^{16}a_{i,k}a_{j,k}p_{k},\\
	I^{2}_{ij}&=\int_{\omega\in \{S_{1},S_{2},S_{3},S_{4}\}}
	b_{i}(\omega )\cdot 
	b_{j}(\omega)p(\eta,\omega )=\sum_{k=1}^{4}b_{i,k}b_{j,k}\tilde{p}_{k},\\
	I^{3}_{ij}&=\int_{\omega\in \{S_{1},S_{2},S_{3},S_{4}\}}
	b_{i}(\omega )\cdot 
	p(\eta,\omega )\cdot \int_{\omega\in \{S_{1},S_{2},S_{3},S_{4}\}}
	b_{j}(\omega)p(\eta,\omega )=\sum_{k=1}^{4}b_{i,k}\tilde{p}_{k}\cdot \sum_{k=1}^{4}b_{j,k}\tilde{p}_{k}.
\end{align*}
Let
\begin{align*}
	A=\left(\begin{array}{cccccccccc} 
		a_{1,1} & a_{1,2}& \cdots & a_{1,16}\\
		a_{2,1} & a_{2,2}& \cdots & a_{2,16}\\
		a_{3,1} & a_{3,2}& \cdots & a_{3,16}\\
		a_{4,1} & a_{4,2}& \cdots & a_{4,16}\end{array}\right ). 
\end{align*}
Then $I^{1}(n,K=0)(S_{1}),\ I^{2}(n,K=0)(S_{1}),\ I^{3}(n,K=0)(S_{1})$ can be written as
\begin{align}\label{5.6}
	I^{1}(n,K=0)(S_{1})&=A \cdot  diag\{p_{1},p_{2},\cdots,p_{16}\}\cdot A^{T},\nonumber\\
	I^{2}(n,K=0)(S_{1})&=B \cdot  diag\{\tilde{p}_{1},\tilde{p}_{2},\tilde{p}_{3},\tilde{p}_{4}\} \cdot B^{T},
	I^{3}(n,K=0)(S_{1})=B \cdot \tilde{p}\cdot \tilde{p}^{T}\cdot
	B^{T}.
\end{align}
Noticed that there exists the following relationship between $A$ and $B$  
\begin{align}\label{eq:4.8}
	A=B\cdot D,
\end{align}
where
\begin{align*}
	D=\left(\begin{array}{cccccccccccccccc} 
		0&-1&-1&-1 &2&1&1&1&2&1&1&1&2&1&1&1\\
		0&1&0&0 &-2&-1&-2&-2&0&1&0&0&0&1&0&0\\
		0&0&1&0 &0&0&1&0&-2&-2&-1&-2&0&0&1&0\\
		0&0&0&1&0&0&0&1&0&0&0&1&-2&-2&-2&-1
	\end{array}\right ) .
\end{align*}
Hence
\begin{align*}
	I^{1}(n,K=0)(S_{1})=B\cdot D \cdot diag\{p_{1},p_{2},\cdots,p_{16}\} \cdot D^{T}\cdot
	B.
\end{align*}
Denote that $D=(D_{1},D_{2},D_{3},D_{4})$ and $diag\{p_{1},p_{2},\cdots,p_{16}\}=diag\{\tilde{P_{1}},\tilde{P_{2}},\tilde{P_{3}},\tilde{P_{4}}\}$. Then
\begin{align*}
	B^{-1}M(n,K=0)(B^{T})^{-1}&=B^{-1}(I^{1}-2I^{2}+2I^{3})(B^{T})^{-1}\\
	&=\sum_{k=1}^{4}D_{k} \cdot \tilde{P_{k}}\cdot D_{k}^{T}
	-2\left(\begin{array}{cccccccccc} 
		\tilde{p}_{1}- \tilde{p}^{2}_{1}&- \tilde{p}_{1}\tilde{p}_{2}& -\tilde{p}_{1}\tilde{p}_{3}& -\tilde{p}_{1}\tilde{p}_{4}\\
		-\tilde{p}_{1}\tilde{p}_{2} &\tilde{p}_{2}- \tilde{p}^{2}_{2}&-\tilde{p}_{2}\tilde{p}_{3}&-\tilde{p}_{2}\tilde{p}_{4}\\
		- \tilde{p}_{1}\tilde{p}_{3} &-\tilde{p}_{2}\tilde{p}_{3}&\tilde{p}_{3}- \tilde{p}^{2}_{3}&-\tilde{p}_{3}\tilde{p}_{4} \\
		-\tilde{p}_{1}\tilde{p}_{4}&-\tilde{p}_{2}\tilde{p}_{4}&-\tilde{p}_{3}\tilde{p}_{4}&\tilde{p}_{4}- \tilde{p}^{2}_{4}\end{array}\right ).
\end{align*}
Obviously
\begin{align*}
	Q_{1}\cdot \left(\begin{array}{cccccccccc} 
		\tilde{p}_{1}- \tilde{p}^{2}_{1}&- \tilde{p}_{1}\tilde{p}_{2}& -\tilde{p}_{1}\tilde{p}_{3}& -\tilde{p}_{1}\tilde{p}_{4}\\
		-\tilde{p}_{1}\tilde{p}_{2} &\tilde{p}_{2}- \tilde{p}^{2}_{2}&-\tilde{p}_{2}\tilde{p}_{3}&-\tilde{p}_{2}\tilde{p}_{4}\\
		- \tilde{p}_{1}\tilde{p}_{3} &-\tilde{p}_{2}\tilde{p}_{3}&\tilde{p}_{3}- \tilde{p}^{2}_{3}&-\tilde{p}_{2}\tilde{p}_{3} \\
		-\tilde{p}_{1}\tilde{p}_{4}&-\tilde{p}_{2}\tilde{p}_{4}&-\tilde{p}_{3}\tilde{p}_{4}&\tilde{p}_{4}- \tilde{p}^{2}_{4}\end{array}\right )\cdot Q_{1}^{T}=\left(\begin{array}{cccccccccc} 
		0&0& 0& 0\\
		0 &\tilde{p}_{2}- \tilde{p}^{2}_{2}&-\tilde{p}_{2}\tilde{p}_{3}&-\tilde{p}_{2}\tilde{p}_{4}\\
		0 &-\tilde{p}_{2}\tilde{p}_{3}&\tilde{p}_{3}- \tilde{p}^{2}_{3}&-\tilde{p}_{3}\tilde{p}_{4} \\
		0&-\tilde{p}_{2}\tilde{p}_{4}&-\tilde{p}_{3}\tilde{p}_{4}&\tilde{p}_{4}- \tilde{p}^{2}_{4}\end{array}\right ).
\end{align*}
Denote
\begin{align*}
	M_{2,2}&=4p_{5}+4p_{7}+4p_{8}+p_{2}+p_{6}+p_{10}+p_{14}=2\Phi^{3}(z)\Phi(-z)+6\Phi(z)\Phi^{3}(-z)+8\Phi^{2}(z)\Phi^{2}(-z),\\
	M_{3,3}&=4p_{9}+4p_{10}+4p_{12}+p_{3}+p_{7}+p_{11}+p_{15}=2\Phi^{3}(z)\Phi(-z)+6\Phi(z)\Phi^{3}(-z)+8\Phi^{2}(z)\Phi^{2}(-z),\\
	M_{4,4}&=4p_{13}+4p_{14}+4p_{15}+p_{4}+p_{8}+p_{12}+p_{16}=4\Phi^{2}(z)\Phi^{2}(-z)+4\Phi^{4}(-z)+8\Phi(z)\Phi^{3}(-z)=4\Phi^{2}(-z),
\end{align*}
where $z=1-\gamma$. Hence
\begin{align*}
	&M^{*}_{1}(n,K=0)(\eta=S_{1})=Q_{1}B^{-1}M(n,K=0)(\eta=S_{1})(B^{T})^{-1} Q_{1}^{T}\\
	&=\left(\begin{array}{cccccccccccccccc} 
		0&0&0&0\\
		0&M_{2,2}&-2p_{7}-2p_{10}& -2p_{8}-2p_{14}\\
		0&-2p_{7}-2p_{10}&M_{3,3}&-2p_{12}-2p_{15} \\
		0&-2p_{8}-2p_{14}&-2p_{12}-2p_{15}&M_{4,4}
	\end{array}\right )-2
	\left(\begin{array}{cccccccccc} 
		0&0& 0& 0\\
		0 &\tilde{p}_{2}- \tilde{p}^{2}_{2}&-\tilde{p}_{2}\tilde{p}_{3}&-\tilde{p}_{2}\tilde{p}_{4}\\
		0 &-\tilde{p}_{2}\tilde{p}_{3}&\tilde{p}_{3}- \tilde{p}^{2}_{3}&-\tilde{p}_{3}\tilde{p}_{4} \\
		0&-\tilde{p}_{2}\tilde{p}_{4}&-\tilde{p}_{3}\tilde{p}_{4}&\tilde{p}_{4}- \tilde{p}^{2}_{4}\end{array}\right )\\
	&=\left(\begin{array}{cccccccccccccccc} 
		0&0&0&0\\
		0&4\Phi(z)\Phi^{2}(-z)+2\Phi^{2}(z)\Phi^{2}(-z)&2\Phi^{2}(z)\Phi^{2}(-z)-4\Phi(z)\Phi^{3}(-z)& -2\Phi^{2}(z)\Phi^{2}(-z)\\
		0&2\Phi^{2}(z)\Phi^{2}(-z)-4\Phi(z)\Phi^{3}(-z)&4\Phi(z)\Phi^{2}(-z)+2\Phi^{2}(z)\Phi^{2}(-z)&-2\Phi^{2}(z)\Phi^{2}(-z)\\
		0&-2\Phi^{2}(z)\Phi^{2}(-z)&-2\Phi^{2}(z)\Phi^{2}(-z)&2\Phi^{2}(-z)+2\Phi^{4}(-z)
	\end{array}\right ).
\end{align*}

\begin{lemma}\label{lem:4.5}
	Assume that the system parameter is $K=0$, the corresponding relationship between the matrices $\{N^{*}_{i},
	\	M^{*}_{i}\}_{i=1,2,3,4}$  defined in Lemma \ref{lem:4.4} and the initial datas $\{\eta=S_{i}\}_{i=1,2,3,4}$ are given by
	\begin{align*}
		N^{*}_{i}(n,K=0)(\eta=S_{i})&=P_{1i}P_{jk}N^{*}_{1}(n,K=0)(\eta=S_{1})P_{jk}P_{1i},\\
		M^{*}_{i}(n,K=0)(\eta=S_{i})&=P_{1i}P_{jk}M^{*}_{1}(n,K=0)(\eta=S_{1})P_{jk}P_{1i} .
	\end{align*}	
	where $i,j,k=2,3,4$, but are not equal to each other. To say it in another way, for $i=2,3,4$, $N^{*}_{i}(n,K=0)(\eta=S_{i})$  is  similar to $N^{*}_{1}(n,K=0)(\eta=S_{1})$ and $M^{*}_{i}(n,K=0)(\eta=S_{i})$  is  similar to $M^{*}_{1}(n,K=0)(\eta=S_{1})$.
\end{lemma}	
\begin{proof}
	Without loss of generality, consider the case $\eta =S_{3}$, and the result for $\eta =S_{2},S_{4}$ also holds by the same token. By definition and Lemma \ref{lem:4.1}
	\begin{align*}
		B^{-1}N(n,K&=0)(\eta=S_{3})(B^{T})^{-1}=P_{13} D_{1} P_{13} P_{13}P_{24}diag\{\tilde{p}_{1},\tilde{p}_{2},\tilde{p}_{3},\tilde{p}_{4}\} P_{24} P_{13}P_{13} D_{1}^{T}P_{13} ,
	\end{align*}
	Then
	\begin{align*}
		N^{*}_{3}(n,K=0)(\eta=S_{3})&=Q_{3}B^{-1}N(n,K=0)(\eta=S_{3})(B^{T})^{-1}Q^{T}_{3}\\
		&=P_{13}Q_{1}P_{13}B^{-1}N(n,K=0)(\eta=S_{3})(B^{T})^{-1}P_{13}Q^{T}_{1}P_{13}\\
		&=diag\{\tilde{p}_{3},\tilde{p}_{4},0,\tilde{p}_{2}\}=P_{13}P_{24}N^{*}_{1}(n,K=0)(\eta=S_{1})P_{24}P_{13}.
	\end{align*}
	Similarly, due to~\eqref{eq:3.4}~and~\eqref{eq:4.5}~
	\begin{align*}
		M^{*}_{3}(n,K=0)(\eta=S_{3})&=Q_{3}B^{-1}M(n,K=0)(\eta=S_{3})(B^{T})^{-1}Q^{T}_{3}\\
		&=Q_{3}B^{-1}(I^{1}(n,K=0)(S_{3})-2I^{2}(n,K=0)(S_{3})+2I^{3}(n,K=0)(S_{3}))(B^{T})^{-1}Q^{T}_{3},
	\end{align*}
	where
	\begin{align*}
		B^{-1}I^{1}(n,K=0)(S_{3})(B^{T})^{-1}&=\tilde{p}_{1}P_{24}P_{13}D_{1}P_{24}P_{13}P_{3}P_{24}P_{13}D^{T}_{1}P_{13}P_{24}\\
		&+\tilde{p}_{2}P_{24}P_{13}D_{2}P_{24}P_{13}P_{4}P_{24}P_{13}D^{T}_{2}P_{13}P_{24}\\
		&+\tilde{p}_{3}P_{24}P_{13}D_{3}P_{13}P_{24}P_{1}P_{24}P_{13}D^{T}_{3}P_{13}P_{24}\\
		&+\tilde{p}_{4}P_{24}P_{13}D_{4}P_{13}P_{24}P_{2}P_{24}P_{13}D^{T}_{4}P_{13}P_{24},
	\end{align*}
	and
	\begin{align*}
		B^{-1}I^{2}(n,K=0)(S_{3})(B^{T})^{-1}&=P_{13}P_{24} diag\{\tilde{p}_{1},\tilde{p}_{2},\tilde{p}_{3},\tilde{p}_{4}\} P_{13}P_{24},\\
		B^{-1}I^{3}(n,K=0)(S_{3})(B^{T})^{-1}&=P_{13}P_{24}\cdot 
		\tilde{p}\cdot  
		\tilde{p}^{T}\cdot P_{13}P_{24}.
	\end{align*}
	According to Lemma \ref{lem:4.1} and equation~\eqref{5.6}~ 
	\begin{align*}
		M^{*}_{3}(n,K=0)(\eta=S_{3})&=P_{13}P_{24}M_{1}^{*}(n,K=0)(\eta=S_{1})P_{13}P_{24}.
	\end{align*}
	This completes the proof.
\end{proof}

\begin{theorem}\label{thm:4.6}
	In the setup above, for arbitrary two-points functions $f$, let $F(n,\eta)=P_{n}f(\eta)$. Then there exists the following curvature estimation
	\begin{align}\label{eq:4.9}
		\Gamma_{2}(F,F)(\eta)\geq \frac{\rho}{2}\Gamma(F,F)(\eta), 
	\end{align}
	where $\eta\in\{+1,-1\}^{2}$ is the initial data at spatial coordinates $x_{1}$ and $x_{2}$. Furthermore, Bakry- \'Emery curvature $\rho$ satisfies
	$$\rho=\frac{\lambda^{M_{1}^{*}(n,K=0)}_{2}}{\lambda^{N_{1}^{*}(n,K=0)}_{4}},$$
	which is independent to $f$ and $t$.
\end{theorem}
\begin{proof}
	Thanks to~\eqref{eq:3.3}~and~\eqref{eq:3.4}~, given $\eta=S_{1}=(+1,+1)$, 
	\begin{align*}
		2\Gamma(F,F)(\eta)=(\vec{f})^{T}
		\cdot N(n,K)(\eta)
		\cdot \vec{f},
		4\Gamma_{2}(F,F)(\eta)=(\vec{f})^{T}
		\cdot M(n,K)(\eta)
		\cdot \vec{f}.
	\end{align*}
	Because $K=0$, then
	\begin{align*}
		&M_{1}^{*}(n,K=0)(\eta=S_{1})=Q_{1}B^{-1}M(n,K=0)(\eta=S_{1})(Q_{1}B^{-1})^{T}\\
		&=\left(\begin{array}{cccccccccccccccc} 
			0&0&0&0\\
			0&4\Phi(z)\Phi^{2}(-z)+2\Phi^{2}(z)\Phi^{2}(-z)&2\Phi^{2}(z)\Phi^{2}(-z)-4\Phi(z)\Phi^{3}(-z)& -2\Phi^{2}(z)\Phi^{2}(-z)\\
			0&2\Phi^{2}(z)\Phi^{2}(-z)-4\Phi(z)\Phi^{3}(-z)&4\Phi(z)\Phi^{2}(-z)+2\Phi^{2}(z)\Phi^{2}(-z)&-2\Phi^{2}(z)\Phi^{2}(-z)\\
			0&-2\Phi^{2}(z)\Phi^{2}(-z)&-2\Phi^{2}(z)\Phi^{2}(-z)&2\Phi^{2}(-z)+2\Phi^{4}(-z)
		\end{array}\right ),
	\end{align*}
	and
	\begin{align*}
		N_{1}^{*}(n,K=0)(\eta=S_{1})=\left(\begin{array}{cccccccccc} 
			0 &0&0&0\\
			0&\Phi(z)\Phi(-z)&0&0\\	
			0&0&\Phi(z)\Phi(-z)&0\\
			0&0&0&\Phi^{2}(-z)\end{array}\right ).
	\end{align*}
	let $\lambda^{M}_{1}\leq \lambda^{M}_{2}\leq\lambda^{M}_{3}\leq\lambda^{M}_{4}$ be the eigenvalues of matrix $M$  in a  ascending order. Then the Fiedler value $\lambda^{M_{1}^{*}(n,K=0)(\eta=S_{1})}_{2}$ of matrix $M_{1}^{*}(n,K=0)(\eta=S_{1})$ is the smallest root of the following quadratic equation
	\begin{align*}
		\lambda^{2}&-(8\Phi^{2}(z)\Phi^{2}(-z)+2\Phi^{2}(-z)+2\Phi^{4}(-z))\lambda\\
		&+8\Phi^{2}(z)\Phi^{2}(-z)\cdot (2\Phi^{2}(-z)+2\Phi^{4}(-z))-8\Phi^{4}(z)\Phi^{4}(-z)=0.
	\end{align*}
	Due to $8\Phi^{2}(z)\Phi^{2}(-z)\cdot (2\Phi^{2}(-z)+2\Phi^{4}(-z))-8\Phi^{4}(z)\Phi^{4}(-z)>0$, hence $\lambda^{M_{1}^{*}(n,K=0)(\eta=S_{1})}_{2}>\lambda^{M_{1}^{*}(n,K=0)(\eta=S_{1})}_{1}=0$, which means
	$$0=\lambda^{M_{1}^{*}(n,K=0)(\eta=S_{1})}_{1}< \lambda^{M_{1}^{*}(n,K=0)(\eta=S_{1})}_{2}\leq\lambda^{M_{1}^{*}(n,K=0)(\eta=S_{1})}_{3}\leq\lambda^{M_{1}^{*}(n,K=0)(\eta=S_{1})}_{4},$$
	$$0=\lambda^{N_{1}^{*}(n,K=0)(\eta=S_{1})}_{1}< \lambda^{N_{1}^{*}(n,K=0)(\eta=S_{1})}_{2}\leq\lambda^{N_{1}^{*}(n,K=0)(\eta=S_{1})}_{3}\leq\lambda^{N_{1}^{*}(n,K=0)(\eta=S_{1})}_{4}.$$
	Further, owing to Lemma \ref{lem:4.5} and the property that matrices which are similar to each other share the comment eigenvalues, then
	\begin{align}\label{eq:4.10}
		\frac{\lambda^{M_{1}^{*}(n,K=0)(\eta=S_{1})}_{2}}{\lambda^{N_{1}^{*}(n,K=0)(\eta=S_{1})}_{4}}=\frac{\lambda^{M_{2}^{*}(n,K=0)(\eta=S_{2})}_{2}}{\lambda^{N_{2}^{*}(n,K=0)(\eta=S_{2})}_{4}}=\frac{\lambda^{M_{3}^{*}(n,K=0)(\eta=S_{3})}_{2}}{\lambda^{N_{3}^{*}(n,K=0)(\eta=S_{3})}_{4}}=\frac{\lambda^{M_{4}^{*}(n,K=0)(\eta=S_{4})}_{2}}{\lambda^{N_{4}^{*}(n,K=0)(\eta=S_{4})}_{4}},
	\end{align}
	Claim: The Bakry- \'Emery curvature $\rho$ in equation~\eqref{eq:4.9}~is independent to $\eta$. According to~\eqref{eq:4.10}~,
	choosing
	$$\rho=\frac{\lambda^{M_{1}^{*}(n,K=0)(\eta=S_{1})}_{2}}{\lambda^{N_{1}^{*}(n,K=0)(\eta=S_{1})}_{4}}.$$ 
	Then
	\begin{align*}
		&M_{i}^{*}(n,K=0)(\eta=S_{i})-\rho N_{i}^{*}(n,K=0)(\eta=S_{i})\\
		&=M_{i}^{*}(n,K=0)(\eta=S_{i})-\lambda^{M_{i}^{*}(n,K=0)(\eta=S_{i})}_{2}I\\
		&+\lambda^{M_{i}^{*}(n,K=0)(\eta=S_{i})}_{2}I-\rho N_{i}^{*}(n,K=0)(\eta=S_{i})\geq 0.
	\end{align*}
	Hence
	\begin{align*}
		&BQ_{i}^{-1}M_{i}^{*}(n,K=0)(\eta=S_{i})(BQ_{i}^{-1})^{T}-\rho BQ_{i}^{-1}N_{i}^{*}(n,K=0)(\eta=S_{i})(BQ_{i}^{-1})^{T}\\
		&=BQ_{i}^{-1}Q_{i}B^{-1}M(n,K=0)(\eta=S_{i})(Q_{i}B^{-1})^{T}(BQ_{i}^{-1})^{T}\\
		&-\rho BQ_{i}^{-1}Q_{i}B^{-1}N(n,K=0)(\eta=S_{i})(Q_{i}B^{-1})^{T}(BQ_{i}^{-1})^{T}\\
		&=M(n,K=0)(\eta=S_{i})-\rho N(n,K=0)(\eta=S_{i})\geq 0.
	\end{align*}
	that is to say, for any $\eta=S_{i},i=1,2,3,4$, there exists $\rho>0$ such that $M(n,K=0)(\eta)-\rho N(n,K=0)(\eta)\geq 0$, then  $\Gamma_{2}(F,F)(\eta)\geq \frac{\rho}{2}\Gamma(F,F)(\eta)$, where $\rho $ is independent to $\eta$ and $f$, the claim is proved.
\end{proof}
\subsection{Continuous time Markov process and local Poincar\'e inequality}\label{sec:4.2}
Based on the discussion about the discrete time Markov chain $X_{n}$ above, in this section, we will focus on it's continuation by compositing a Possion process and establish the local Poincar\'e inequality.
\begin{proposition}\label{pro:4.7}
	Consider a general system with parameter $K\neq 0$, there admits the following decomposition 
	\begin{align*}
		M(T-t,K)(\eta )&=M^{1}(T-t,K)(\eta)+M^{2}(T-t,K)(\eta),\\
		N(T-t,K)(\eta )&=N^{1}(T-t,K)(\eta)+N^{2}(T-t,K)(\eta).
	\end{align*}
	where $M^{1}(T-t,K)(\eta)$ and $N^{1}(T-t,K)(\eta)$ only depend on the finite dimensional projection of configuration $\eta$. For arbitrary $\varepsilon >0$, the elements of $M^{2}(T-t,K)(\eta)$ and $N^{2}(T-t,K)(\eta)$ satisfy
	\begin{align*}
		|M_{kl}^{2}(T-t,K)(\eta)|\leq 8\varepsilon,\
		|N_{kl}^{2}(T-t,K)(\eta)|\leq 2\varepsilon,\ \text{for}\ k,l=1,2,3,4.
	\end{align*}
	where $\varepsilon$ is  independent to both system parameter $K$ and initial data $\eta$.
\end{proposition} 
\begin{proof}
	Let $p(n,\eta,\omega)$ be the $n$-step transition probability of the Markov chain  $X_{n}$, then the transition probability of $X_{t}$ satisfies
	\begin{align*}
		p(t,\eta,\omega)=\sum_{k=0}^{\infty}e^{- t}\frac{t^{k}}{k!}p(k,\eta,\omega).
	\end{align*}
	Hence, for any fixed time $T-t, t\in[0,T]$, $\forall  \varepsilon >0$, there exists a positive integer large enough $N^{\varepsilon}_{T-t}$ such that
	\begin{align*}
		\sum_{k=N^{\varepsilon}_{T-t}+1}^{\infty}e^{-(T-t)}\frac{(T-t)^{k}}{k!}<\varepsilon.
	\end{align*}
	then
	\begin{align*}
		p(T-t,\eta,\omega)&=\sum_{k=0}^{N^{\varepsilon}_{T-t}}e^{-(T-t)}\frac{(T-t)^{k}}{k!}p(k,\eta,\omega)+\sum_{k=N^{\varepsilon}_{T-t}+1}^{\infty}e^{-(T-t)}\frac{(T-t)^{k}}{k!}p(k,\eta,\omega)\\
		&<\sum_{k=0}^{N^{\varepsilon}_{T-t}}e^{-(T-t)}\frac{(T-t)^{k}}{k!}p(k,\eta,\omega)+\varepsilon,\ \quad  \forall K,\ \forall \eta.
	\end{align*}
	Due to \eqref{eq:3.4}
	\begin{align*}
		M(T-t,K)(\eta)&=\int_{\omega\in E}\int_{z\in E}\mathbb{A}(z,\omega,\eta) \cdot \mathbb{A}^{T}(z,\omega,\eta)p(\omega,dz )p(\eta,d\omega )\\
		&-2\int_{\omega\in E}\mathbb{B}(\omega) \cdot \mathbb{B}^{T}(\omega)p(\eta,d\omega )+2\int_{\omega\in E}\mathbb{B}(\omega)p(\eta,d\omega ) \cdot \int_{\omega\in E}\mathbb{B}^{T}(\omega)p(\eta,d\omega ),
	\end{align*}
	where
	\begin{align*}
		b_{i}(\omega)&=p(T-t,\omega,S_{i})=  \sum_{k=0}^{N^{\varepsilon}_{T-t}}e^{-(T-t)}\frac{(T-t)^{k}}{k!}p(k,\omega,S_{i})+\sum_{k=N^{\varepsilon}_{T-t}+1}^{\infty}e^{-(T-t)}\frac{(T-t)^{k}}{k!}p(k,\omega,S_{i})
		\\
		&:= b^{1}_{i}(\omega)+b^{2}_{i}(\omega),\quad 0<b^{1}_{i}(\omega)\leq 1, \quad 0<b^{2}_{i}(\omega)<\varepsilon,\\
		a_{i}(z,\omega,\eta)
		&=b_{i}(z)-2b_{i}(\omega)+b_{i}(\eta)\\
		&=b^{1}_{i}(z)-2b^{1}_{i}(\omega)+b^{1}_{i}(\eta)+b^{2}_{i}(z)-2b^{2}_{i}(\omega)+b^{2}_{i}(\eta)\\
		&:= a^{1}_{i}(\omega)+a^{2}_{i}(\omega),\quad -2<a^{1}_{i}(\omega)\leq 2, \quad -2\varepsilon<a^{2}_{i}(\omega)<2\varepsilon.
	\end{align*}
	That is to say, the transition probability can be divided into 
	\begin{align*}
		\mathbb{A}(z,\omega,\eta)=\mathbb{A}^{1}(z,\omega,\eta)+\mathbb{A}^{2}(z,\omega,\eta), \ \text{and}\
		\mathbb{B}(\omega)=\mathbb{B}^{1}(\omega)+\mathbb{B}^{2}(\omega).
	\end{align*}
	With this decomposition in hand, we can also divide matrix $M(T-t,K)(\eta)$ into two parts
	\begin{align*}
		M(T-t,K)(\eta)=M^{1}(T-t,K)(\eta)+M^{2}(T-t,K)(\eta ),
	\end{align*}
	where
	\begin{align}\label{eq:5.9}
		M^{1}(T-t,K)(\eta)&=\int_{\omega\in E}\int_{z\in E}\mathbb{A}^{1}(z,\omega,\eta) \cdot \mathbb{A}^{1}(z,\omega,\eta)^{T}p(\omega,dz )p(\eta,d\omega )\nonumber\\
		&-2\int_{\omega\in E}\mathbb{B}^{1}(\omega) \cdot \mathbb{B}^{1}(\omega)^{T}p(\eta,d\omega )\\
		&+2\int_{\omega\in E}\mathbb{B}^{1}(\omega)p(\eta,d\omega ) \cdot \int_{\omega\in E}\mathbb{B}^{1}(\omega)^{T}p(\eta,d\omega ).\nonumber
	\end{align}
	and
	\begin{align*}
		|M_{kl}^{2}(T-t,K)(\eta)|\leq 8\varepsilon.
	\end{align*}
	Similarly
	\begin{align*}
		N(T-t,K)(\eta)&=N^{1}(T-t,K)(\eta)+N^{2}(T-t,K)(\eta),
	\end{align*}
	where
	\begin{align}\label{eq:4.11}
		N^{1}(T-t,K)(\eta)=\int_{\omega\in E}(\mathbb{B}^{1}(\omega)-\mathbb{B}^{1}(\eta)) \cdot (\mathbb{B}^{1}(\omega)-\mathbb{B}^{1}(\eta))^{T}p(\eta,d\omega ),
	\end{align}
	and
	\begin{align*}
		|N_{kl}^{2}(T-t,K)(\eta)|\leq 2\varepsilon.
	\end{align*}
	which completes the lemma.
\end{proof}
\begin{lemma}\label{lem:4.8} 
	For the case that system parameter $K= 0$, there exists the invertible matrix  $\hat{B}(T-t)$ such that
	\begin{align*}
		\hat{B}(T-t)^{-1}M^{1}(T-t,K=0)(\eta=S_{1})(\hat{B}(T-t)^{T})^{-1}&=D \cdot  diag\{p_{1},p_{2},\cdots,p_{16}\}\cdot D^{T}\nonumber\\
		&-2\cdot diag\{\tilde{p}_{1},\tilde{p}_{2},\tilde{p}_{3},\tilde{p}_{4}\} +2 \tilde{p}\cdot \tilde{p}^{T},\\
		\hat{B}(T-t)^{-1}N^{1}(T-t,K=0)(\eta=S_{1})(\hat{B}(T-t)^{T})^{-1}&=\hat{P}_{1}.
	\end{align*}
	where $p_{i},i=1,\cdots, 16$ and  $\tilde{p}=(\tilde{p}_{1},\tilde{p}_{2},\tilde{p}_{3},\tilde{p}_{4})$ defined as in equation \eqref{eq:4.3}.
\end{lemma}
\begin{proof}
	According to equation~\eqref{eq:4.11}~ in proposition \ref{pro:4.7}, we know that
	\begin{align*}
		N^{1}(T-t,K)(\eta)=\int_{\omega\in E}(\mathbb{B}^{1}(\omega)-\mathbb{B}^{1}(\eta)) \cdot (\mathbb{B}^{1}(\omega)-\mathbb{B}^{1}(\eta))^{T}p(\eta,d\omega ),
	\end{align*}
	Hence
	\begin{align*}
		(N^{1}(T-t,K)(\eta))_{ij}=\int_{\omega\in E}&\sum_{k=0}^{N^{\varepsilon}_{T-t}}e^{-(T-t)}\frac{(T-t)^{k}}{k!}(p(k,\omega,S_{i})-p(k,\eta,S_{i})) \\
		&\cdot \sum_{k=0}^{N^{\varepsilon}_{T-t}}e^{-(T-t)}\frac{(T-t)^{k}}{k!}(p(k,\omega,S_{j})-p(k,\eta,S_{j}))p(\eta,d\omega ).
	\end{align*}
	When $K= 0$, then 	
	\begin{align*}
		(N^{1}(T-t,K=0)(\eta))_{ij}=\int_{\omega\in E_{2}}&\sum_{k=0}^{N^{\varepsilon}_{T-t}}e^{-(T-t)}\frac{(T-t)^{k}}{k!}(p(k,\omega,S_{i})-p(k,\eta,S_{i})) \\
		&\cdot \sum_{k=0}^{N^{\varepsilon}_{T-t}}e^{-(T-t)}\frac{(T-t)^{k}}{k!}(p(k,\omega,S_{j})-p(k,\eta,S_{j}))p(\eta,d\omega ).
	\end{align*}
	Without loss of generality, assume that $\eta=S_{1}$, then
	\begin{align*}
		(N^{1}(T-t,K=0)(\eta))_{ij}=\sum_{l=1}^{4}&\{\sum_{k=0}^{N^{\varepsilon}_{T-t}}e^{-(T-t)}\frac{(T-t)^{k}}{k!}(b_{i,l}(k)-b_{i,1}(k))\}\\
		&\cdot \{\sum_{k=0}^{N^{\varepsilon}_{T-t}}e^{-(T-t)}\frac{(T-t)^{k}}{k!}(b_{j,l}(k)-b_{j,1}(k))\}\tilde{p}_{l}.
	\end{align*}
	Using the matrix $B(n)$ introduced in Lemma \ref{lem:4.2}, define
	\begin{align*}
		\hat{B}(T-t):= \sum_{k=0}^{N^{\varepsilon}_{T-t}}e^{-(T-t)}\frac{(T-t)^{k}}{k!}B(k).
	\end{align*}
	Recall that
	\begin{align*}
		D_{1}=
		\left(\begin{array}{cccccccccc} 
			0 & -1& -1 &-1\\
			0 & 1& 0 &0\\
			0 & 0& 1 & 0\\
			0 & 0& 0 & 1\end{array}\right ).
	\end{align*}
	Then $N(T-t,K=0)(\eta)$ has the following form
	\begin{align*}
		N^{1}(T-t,K=0)(\eta)=&	\hat{B}(T-t) D_{1}diag\{\tilde{p}_{1},\tilde{p}_{2},\tilde{p}_{3},\tilde{p}_{4}\} D_{1}^{T}\hat{B}(T-t)^{T}\\
		&=\hat{B}(T-t)\hat{P}_{1}\hat{B}(T-t)^{T}.
	\end{align*}
	As we will talk in Lemma \ref{lem:4.9} below,  $\hat{B}(T-t)$ is invertible, hence
	\begin{align*}
		\hat{B}(T-t)^{-1}N^{1}(T-t,K=0)(\hat{B}(T-t)^{T})^{-1}=\hat{P}_{1}.
	\end{align*}	
	Similarly, for
	\begin{align*}
		M^{1}(T-t,K)(\eta)&=\int_{\omega\in E}\int_{z\in E}\mathbb{A}^{1}(z,\omega,\eta) \cdot \mathbb{A}^{1}(z,\omega,\eta)^{T}p(\omega,dz )p(\eta,d\omega )\nonumber\\
		&-2\int_{\omega\in E}\mathbb{B}^{1}(\omega) \cdot \mathbb{B}^{1}(\omega)^{T}p(\eta,d\omega )\\
		&+2\int_{\omega\in E}\mathbb{B}^{1}(\omega)p(\eta,d\omega ) \cdot \int_{\omega\in E}\mathbb{B}^{1}(\omega)^{T}p(\eta,d\omega ).\nonumber
	\end{align*}
	When $K=0$ and $\eta=S_{1}$, define
	\begin{align*}
		\hat{A}(T-t):= \sum_{k=0}^{N^{\varepsilon}_{T-t}}e^{-(T-t)}\frac{(T-t)^{k}}{k!}A(k).
	\end{align*}	
	Due to~\eqref{eq:4.8}~
	\begin{align*}
		A(k)=B(k)\cdot D,
	\end{align*}
	where
	\begin{align*}
		D=\left(\begin{array}{cccccccccccccccc} 
			0&-1&-1&-1 &2&1&1&1&2&1&1&1&2&1&1&1\\
			0&1&0&0 &-2&-1&-2&-2&0&1&0&0&0&1&0&0\\
			0&0&1&0 &0&0&1&0&-2&-2&-1&-2&0&0&1&0\\
			0&0&0&1&0&0&0&1&0&0&0&1&-2&-2&-2&-1
		\end{array}\right ) .
	\end{align*}
	Explaining $M^{1}(T-t,K=0)(\eta=S_{1})$ in the matrix word
	\begin{align}\label{eq:4.12}
		M^{1}(T-t,K=0)(\eta=S_{1})&=\hat{A}(T-t) \cdot  diag\{p_{1},p_{2},\cdots,p_{16}\}\cdot \hat{A}(T-t)^{T}\nonumber\\
		&-2\hat{B}(T-t) \cdot  diag\{\tilde{p}_{1},\tilde{p}_{2},\tilde{p}_{3},\tilde{p}_{4}\} \cdot \hat{B}(T-t)^{T}\nonumber\\
		&+2\hat{B}(T-t) \cdot \tilde{p}\cdot \tilde{p}^{T}\cdot
		\hat{B}(T-t)^{T}.
	\end{align}
	Using Lemma \ref{lem:4.9} again, $\hat{B}(T-t)$ is invertible, then
	\begin{align*}
		\hat{B}(T-t)^{-1}M^{1}(T-t,K=0)(\eta=S_{1})(\hat{B}(T-t)^{T})^{-1}&=D \cdot  diag\{p_{1},p_{2},\cdots,p_{16}\}\cdot D^{T}\nonumber\\
		&-2\cdot diag\{\tilde{p}_{1},\tilde{p}_{2},\tilde{p}_{3},\tilde{p}_{4}\} +2 \tilde{p}\cdot \tilde{p}^{T}.
	\end{align*}
	We have thus proved the lemma.
\end{proof}
\begin{lemma}\label{lem:4.9}
	For any integer $n\geq 1$, $B(n)$ is the matrix introduced in Lemma \ref{lem:4.2}, and $\hat{B}(T-t)$ is defined as in Lemma \ref{lem:4.8}
	\begin{align*}
		\hat{B}(T-t):= \sum_{k=0}^{N^{\varepsilon}_{T-t}}e^{-(T-t)}\frac{(T-t)^{k}}{k!}B(k),
	\end{align*}
	where $B(0)=I$. Then for any time $T-t\geq 0$, $\hat{B}(T-t)$ is invertible and satisfies
	\begin{align*}
		|\hat{B}(T-t)_{ij}^{-1}|\leq C(T-t)=O(e^{T-t}).
	\end{align*}
\end{lemma}	
\begin{proof}
	Assume $y_{1}(0)=1,y_{2}(0)=0$, due to Lemma \ref{lem:4.2} and  \ref{lem:4.3}, we know that
	\begin{align*}
		\hat{B}(T-t)&=\sum_{k=0}^{N^{\varepsilon}_{T-t}}e^{-(T-t)}\frac{(T-t)^{k}}{k!}\left(\begin{array}{cccccccccc} 
			b_{1,1}(k) & b_{1,2}(k)& b_{1,3}(k) & b_{1,4}(k)\\
			b_{2,1}(k)& b_{2,2}(k)& b_{2,3}(k) & b_{2,4}(k)\\
			b_{3,1}(k) & b_{3,2}(k)& b_{3,3}(k)& b_{3,4}(k)\\
			b_{4,1}(k) &b_{4,2}(k)& b_{4,3}(k) & b_{4,4}(n)\end{array}\right )\\
		&=\sum_{k=0}^{N^{\varepsilon}_{T-t}}e^{-(T-t)}\frac{(T-t)^{k}}{k!}\left(\begin{array}{cccccccccc} 
			y^{2}_{1}(k)& y_{1}(k)y_{2}(k)& y_{1}(k)y_{2}(k) & y^{2}_{2}(k)\\
			y_{1}(k)y_{2}(k) &y^{2}_{1}(k)&  y^{2}_{2}(k) & y_{1}(k)y_{2}(k)\\
			y_{1}(k)y_{2}(k) &  y^{2}_{2}(k)& y^{2}_{1}(k)& y_{1}(k)y_{2}(k)\\
			y^{2}_{2}(k)&y_{1}(k)y_{2}(k)& y_{1}(k)y_{2}(k) &y^{2}_{1}(k)\end{array}\right ),
	\end{align*}
	With some long but straightforward algebra, we get
	\begin{align*}
		det(\hat{B}(T-t))
		&=\sum_{k=0}^{N^{\varepsilon}_{T-t}}e^{-(T-t)}\frac{(T-t)^{k}}{k!}\cdot\sum_{k=0}^{N^{\varepsilon}_{T-t}}e^{-(T-t)}\frac{(T-t)^{k}}{k!} (y_{1}(k)-y_{2}(k))^{2}\\
		&\cdot \left(\sum_{k=0}^{N^{\varepsilon}_{T-t}}e^{-(T-t)}\frac{(T-t)^{k}}{k!}(y_{1}(k)- y_{2}(k))\right)^{2}.
	\end{align*}	
	Thanks to Lemma \ref{lem:4.3}, for any integer $k\geq 1$,   $$y_{1}(k)- y_{2}(k)=2(\Phi(1-\gamma)-\frac{1}{2})(\Phi(1-\gamma)-\Phi(\gamma-1))^{k-1}.$$
	We choose the parameter $\gamma> 1$, so $-1<\Phi(1-\gamma)-\Phi(\gamma-1))<0$. Consider the following series of functions
	$$g_{n}(x)=\sum_{k=0}^{n}\frac{x^{k}}{k!},x\in [t-T,0].$$
	It is easy to know, $g_{n}(x)$ is uniformly convergent to $e^{x}$ in interval $ [t-T,0]$. That is to say, given a  fixed  $T-t\geq 0$, for any $0<\delta<e^{t-T}$, there exists a $N^{\delta}_{T-t}$ large enough, such that for $n\geq N^{\delta}_{T-t}$
	\begin{align*}
		\sup_{x\in [t-T,0]}|g_{n}(x)-e^{x}|\leq \delta,
	\end{align*}	
	which means for $n\geq N^{\delta}_{T-t}$
	\begin{align*}
		0<e^{t-T}-\delta \leq e^{x}-\delta \leq g_{n}(x)\leq \delta+e^{x}.
	\end{align*}
	Without loss of generality, we choose $\delta= \frac{1}{2}e^{t-T}$ and $N^{\varepsilon}_{T-t}\geq N^{\delta}_{T-t}$. Hence
	\begin{align*}
		\sum_{k=0}^{N^{\varepsilon}_{T-t}}e^{-(T-t)}\frac{(T-t)^{k}}{k!}(\Phi(1-\gamma)-\Phi(\gamma-1))^{k}>\frac{1}{2}e^{t-T}>0.
	\end{align*}	
	Consequently, 
	\begin{align*}
		det(\hat{B}(T-t))
		&=\sum_{k=0}^{N^{\varepsilon}_{T-t}}e^{-(T-t)}\frac{(T-t)^{k}}{k!}\cdot\sum_{k=0}^{N^{\varepsilon}_{T-t}}e^{-(T-t)}\frac{(T-t)^{k}}{k!} (y_{1}(k)-y_{2}(k))^{2}\\
		&\cdot \left(\sum_{k=0}^{N^{\varepsilon}_{T-t}}e^{-(T-t)}\frac{(T-t)^{k}}{k!}(y_{1}(k)- y_{2}(k))\right)^{2}>0,
	\end{align*}
	then $\hat{B}(T-t)$ is invertible, this completes the proof.
\end{proof}
\begin{proposition}\label{pro:4.10}
	Similar to what we have talked in proposition \ref{pro:4.7}, for the general system parameter $K$, there also admits the decomposition 
	\begin{align*}
		M^{*}_{i}(T-t,K)(\eta )&=M_{i}^{*1}(T-t,K)(\eta)+M_{i}^{*2}(T-t,K)(\eta),\\
		N^{*}_{i}(T-t,K)(\eta )&=N_{i}^{*1}(T-t,K)(\eta)+N_{i}^{*2}(T-t,K)(\eta),i=1,2,3,4.
	\end{align*}
	where $M_{i}^{*1}(T-t,K)(\eta)$ and $	N_{i}^{*1}(T-t,K)(\eta)$ only depend on the finite dimensional projection of configuration $\eta$. For arbitrary $\varepsilon >0$, $M_{i}^{*}(T-t,K)(\eta)$ has the same structure with $M_{i}^{*2}(T-t,K)(\eta)$, which means the elements in the $i$-th row and $i$-th column of  $M_{i}^{*}(T-t,K)(\eta)$ and $M_{i}^{*2}(T-t,K)(\eta)$ are all zeros. Besides, the elements of $M_{i}^{*2}(T-t,K)(\eta)$ and $N_{i}^{*2}(T-t,K)(\eta)$ satisfy
	\begin{align*}
		-\varepsilon \cdot I_{i} \leq M_{i}^{*2}(T-t,K)(\eta)&=Q_{i}B^{-1}M^{2}(T-t,K)(\eta)(Q_{i}B^{-1})^{T}\leq \varepsilon \cdot I_{i},\\
		-\varepsilon \cdot I_{i} \leq N_{i}^{*2}(T-t,K)(\eta)&=Q_{i}B^{-1}N^{2}(T-t,K)(\eta)(Q_{i}B^{-1})^{T}\leq \varepsilon \cdot I_{i},i=1,2,3,4.
	\end{align*}
	where $\varepsilon >0$ is independent to $K$ and $\eta$. 
	Further, in the case of $K=0$, for $\varepsilon >0$ sufficiently small, the following estimation holds
	$$M_{i}^{*1}(T-t,K=0)(\eta=S_{i})-\varepsilon I_{i}>0,i=1,2,3,4.$$  
\end{proposition}
\begin{proof}
	According to the expansion in proposition \ref{pro:4.7}
	\begin{align*}
		\sum_{l=1}^{4}b^{1}_{l}(\omega)=\sum_{k=0}^{N^{\varepsilon}_{T-t}}e^{-(T-t)}\frac{(T-t)^{k}}{k!}\sum_{l=1}^{4}p(k,\omega,S_{l})=\sum_{k=0}^{N^{\varepsilon}_{T-t}}e^{-(T-t)}\frac{(T-t)^{k}}{k!}.
	\end{align*}
	An argument similar to the one used in section  \ref{sec:3.2} shows that 
	\begin{align*}
		Q_{1}M^{1}(T-t,K)Q_{1}^{T}&=\left(\begin{array}{cccccccccc} 
			0&0&0&0\\
			0&\\
			0&&M^{1}_{rem,1}(T-t,K)\\
			0&\end{array}\right ),
	\end{align*}
	and
	\begin{align*}
		Q_{1}N^{1}(T-t,K)Q_{1}^{T}
		&=\left(\begin{array}{cccccccccc} 
			0&0&0&0\\
			0&\\
			0&&N^{1}_{rem,1}(T-t,K)\\
			0&\end{array}\right ).
	\end{align*}
	Define the following matrices
	\begin{align}\label{4.13}
		M_{i}^{*1}(T-t,K)&=Q_{i}\hat{B}(T-t)^{-1}M^{1}(T-t,K)(Q_{i}\hat{B}(T-t)^{-1})^{T},\nonumber\\
		M_{i}^{*2}(T-t,K)&=Q_{i}\hat{B}(T-t)^{-1}M^{2}(T-t,K)(Q_{i}\hat{B}(T-t)^{-1})^{T},\nonumber\\
		N_{i}^{*1}(T-t,K)&=Q_{i}\hat{B}(T-t)^{-1}N^{1}(T-t,K)(Q_{i}\hat{B}(T-t)^{-1})^{T},\\
		N_{i}^{*2}(T-t,K)&=Q_{i}\hat{B}(T-t)^{-1}N^{2}(T-t,K)(Q_{i}\hat{B}(T-t)^{-1})^{T},i=1,2,3,4.\nonumber
	\end{align}
	By the definition, $\hat{B}(T-t)$ shares the same structure with $B(n)$. As we have talked in Lemma \ref{lem:4.4}, the form of $Q_{1}\hat{B}^{-1}(T-t)Q_{1}^{-1}$ is also equation ~\eqref{eq:4.6}~, then
	\begin{align*}
		M_{1}^{*1}(T-t,K)&=\left(\begin{array}{cccccccccc} 
			0&0&0&0\\
			0&\\
			0&&M_{rem,1}^{*1}(T-t,K)\\
			0&\end{array}\right ),\\
		N_{1}^{*1}(T-t,K)&=\left(\begin{array}{cccccccccc} 
			0&0&0&0\\
			0&\\
			0&&N_{rem,1}^{*1}(T-t,K)\\
			0&\end{array}\right ).
	\end{align*}
	that is to say, the elements in first row and first column of $M_{1}^{*1}(T-t,K)$ and $N_{1}^{*1}(T-t,K)$ are all zeros. The results that elements in $i$-th row and $i$-th column of $M_{i}^{*1}(T-t,K)$ and $N_{i}^{*1}(T-t,K)$ are all zeros for $i=1,2,3,4$ can be proved in the same way. Obviously, according to \eqref{eq:3.5}, the elements in $i$-th row and $i$-th column of $M_{i}^{*}(T-t,K)$ and $N_{i}^{*}(T-t,K)$ are all zeros by the similar argument in Lemma \ref{lem:4.4}.
	
	Hence, $M_{i}^{*1}(T-t,K)(\eta)$ shares  the same structure with $M_{i}^{*}(T-t,K)(\eta)$, and $N_{i}^{*1}(T-t,K)(\eta)$ shares  the same structure with  $N_{i}^{*}(T-t,K)(\eta)$. Thanks to
	\begin{align*}
		M^{*2}_{i}(T-t,K)(\eta )&=M_{i}^{*}(T-t,K)(\eta)-M_{i}^{*1}(T-t,K)(\eta),\\
		N^{*2}_{i}(T-t,K)(\eta )&=N_{i}^{*}(T-t,K)(\eta)-N_{i}^{*1}(T-t,K)(\eta),i=1,2,3,4.
	\end{align*}
	we know that the $i$-th  row and $i$-th  column of $M_{i}^{*2}(T-t,K)(\eta)$ and $N_{i}^{*2}(T-t,K)(\eta)$ are all zeros.
	
	The next thing to do in the proof is to control the terms $M_{i}^{*2}(T-t,K)(\eta)$ and  $N_{i}^{*2}(T-t,K)(\eta),i=1,2,3,4$. According to proposition \ref{pro:4.7}, for given $T-t, t\in[0,T]$ and $\forall \varepsilon >0$, choosing $\varepsilon(T-t)=e^{2(t-T)} \cdot \varepsilon >0$, there exists integer $N^{\varepsilon(T-t)}_{T-t}$ large enough such that
	$$
	|M_{ij}^{2}(T-t,K)(\eta)|\leq 8e^{2(t-T)} \cdot \varepsilon,
	|N_{ij}^{2}(T-t,K)(\eta)|\leq 2e^{2(t-T)} \cdot \varepsilon,$$ 
	By the definition of $\hat{B}(T-t)$ in Lemma \ref{lem:4.9}, $|\hat{B}(T-t)_{ij}^{-1}|\leq C(T-t)=O(e^{T-t})$. Hence
	\begin{align*}
		\vert (M_{i}^{*2}(T-t,K)(\eta))_{kl}\vert&=\vert(Q_{i}B^{-1}M^{2}(T-t,K)(\eta)(Q_{i}B^{-1})^{T})_{kl}\vert\leq C_{1}\varepsilon,\\
		\vert(N_{i}^{*2}(T-t,K)(\eta))_{kl}\vert&=\vert(Q_{i}B^{-1}N^{2}(T-t,K)(\eta)(Q_{i}B^{-1})^{T})_{kl}\vert\leq C_{2}\varepsilon,i=1,2,3,4.
	\end{align*}
	where $C_{1}$, $C_{2}$ are constants. Using the Gersgorin theorem, there exists a constant $C$ which is independent  to $T$ and $\eta$, such that for arbitrary $\varepsilon >0$
	\begin{align*}
		-C\varepsilon \cdot I \leq & M_{i}^{*2}(T-t,K)(\eta)\leq C\varepsilon \cdot I,\\
		-C\varepsilon \cdot I \leq & N_{i}^{*2}(T-t,K)(\eta)\leq C\varepsilon \cdot I,i=1,2,3,4.
	\end{align*}
	Recall that elements in $i$-th row and $i$-th column of $M_{i}^{*2}(T-t,K)(\eta)$  and $N_{i}^{*2}(T-t,K)(\eta)$ are all zeros, we introduce  the following matrices
	\begin{align*}
		I_{1}&=\left(\begin{array}{cccccccccc} 
			0 & 0& 0 &0\\
			0 & 1& 0 &0\\
			0 & 0& 1 & 0\\
			0 & 0& 0 & 1\end{array}\right ),I_{2}=\left(\begin{array}{cccccccccc} 
			1 & 0& 0 &0\\
			0 & 0& 0 &0\\
			0 & 0& 1 & 0\\
			0 & 0& 0 & 1\end{array}\right ),
		I_{3}=\left(\begin{array}{cccccccccc} 
			1 & 0& 0 &0\\
			0 & 1& 0 &0\\
			0 & 0& 0 & 0\\
			0 & 0& 0 & 1\end{array}\right ),	I_{4}=\left(\begin{array}{cccccccccc} 
			1 & 0& 0 &0\\
			0 & 1& 0 &0\\
			0 & 0& 1 & 0\\
			0 & 0& 0 & 0\end{array}\right ).
	\end{align*} 
	Since $\varepsilon >0$ is arbitrary, without loss of generality, we have
	\begin{align*}
		-\varepsilon \cdot I_{i} \leq M_{i}^{*2}(T-t,K)(\eta)\leq \varepsilon \cdot I_{i},\
		-\varepsilon \cdot I_{i} \leq N_{i}^{*2}(T-t,K)(\eta)\leq \varepsilon \cdot I_{i},i=1,2,3,4.
	\end{align*}
	
	Finally, we consider the case $K=0$. Thanks to Lemma \ref{lem:4.8} 
	\begin{align*}
		M_{1}^{*1}(T-t,K=0)(\eta=S_{1})&=Q_{1}\hat{B}(T-t)^{-1}M^{1}(T-t,K=0)(\eta=S_{1})(Q_{1}\hat{B}(T-t)^{-1})^{T}\\
		&=Q_{1}D \cdot  diag\{p_{1},p_{2},\cdots,p_{16}\}\cdot D^{T}Q_{1}^{T}\nonumber\\
		&-2Q_{1}\cdot diag\{\tilde{p}_{1},\tilde{p}_{2},\tilde{p}_{3},\tilde{p}_{4}\}Q_{1}^{T} +2Q_{1} \tilde{p}\cdot \tilde{p}^{T}Q_{1}^{T},
	\end{align*}
	By the calculation in theorem \ref{thm:4.6}, we know that 
	$$0=\lambda^{M_{1}^{*1}(T-t,K=0)(\eta=S_{1})}_{1} <\lambda^{M_{1}^{*1}(T-t,K=0)(\eta=S_{1})}_{2}\leq\lambda^{M_{1}^{*1}(T-t,K=0)(\eta=S_{1})}_{3}\leq\lambda^{M_{1}^{*1}(T-t,K=0)(\eta=S_{1})}_{4},$$
	and
	$$0=\lambda^{N_{1}^{*1}(T-t,K=0)(\eta=S_{1})}_{1}<\lambda^{N_{1}^{*1}(T-t,K=0)(\eta=S_{1})}_{2}\leq\lambda^{N_{1}^{*1}(T-t,K=0)(\eta=S_{1})}_{3}\leq\lambda^{N_{1}^{*1}(T-t,K=0)(\eta=S_{1})}_{4}.$$
	Because of the fact that matrices which are similar to each other share the comment eigenvalues, for $i=2,3,4$ , according to Lemma \ref{lem:4.11} below, we obtain
	$$0=\lambda^{M_{i}^{*1}(T-t,K=0)(\eta=S_{i})}_{1} <\lambda^{M_{i}^{*1}(T-t,K=0)(\eta=S_{i})}_{2}\leq\lambda^{M_{i}^{*1}(T-t,K=0)(\eta=S_{i})}_{3}\leq\lambda^{M_{i}^{*1}(T-t,K=0)(\eta=S_{i})}_{4}.$$
	and
	$$0=\lambda^{N_{i}^{*1}(T-t,K=0)(\eta=S_{i})}_{1}<\lambda^{N_{i}^{*1}(T-t,K=0)(\eta=S_{i})}_{2}\leq\lambda^{N_{i}^{*1}(T-t,K=0)(\eta=S_{i})}_{3}\leq\lambda^{N_{i}^{*1}(T-t,K=0)(\eta=S_{i})}_{4}.$$
	Consequently, there exists  $\varepsilon >0$ sufficiently small, such that
	$$M_{i}^{*1}(T-t,K=0)(\eta=S_{i})-\varepsilon I_{i}>0, \ ie.\lambda^{M_{i}^{*1}(T-t,K=0)(\eta=S_{i})-\varepsilon I_{i}}_{2}>0.$$  
	We have thus proved the proposition.
\end{proof}
\begin{lemma}\label{lem:4.11}
	In the case of $K=0$, there exists the following relationship between the matrices $\{N^{*1}_{i},
	\	M^{*1}_{i}\}_{i=1,2,3,4}$ defined in proposition \ref{pro:4.10} and the initial datas $\{\eta=S_{i}\}_{i=1,2,3,4}$
	\begin{align*}
		N^{*1}_{i}(T-t,K=0)(\eta=S_{i})&=P_{1i}P_{jk}N^{*1}_{1}(T-t,K=0)(\eta=S_{i})P_{jk}P_{1i},\\
		M^{*1}_{i}(T-t,K=0)(\eta=S_{i})&=P_{1i}P_{jk}M^{*1}_{1}(T-t,K=0)(\eta=S_{i})P_{jk}P_{1i} .
	\end{align*}	
	where $i,j,k=2,3,4$ and is not equal to each other. Furthermore, for $i=2,3,4$ and arbitrary  $c\in \mathbb{R}$, $N^{*1}_{i}(T-t,K=0)(\eta=S_{i})-cI_{i}$ is similar to $N^{*1}_{1}(T-t,K=0)(\eta=S_{1})-cI_{1}$, and $M^{*1}_{i}(T-t,K=0)(\eta=S_{i})-cI_{i}$  is similar to $M^{*1}_{1}(T-t,K=0)(\eta=S_{1})-cI_{1}$.
\end{lemma}	
\begin{proof}
	Without loss of generality, consider the case $\eta =S_{3}$, the argument for case $i=2,4$ are same. Noticed that
	\begin{align*}
		\hat{B}(T-t)^{-1}N^{1}(T-t,K=&0)(\eta=S_{3})(\hat{B}(T-t)^{T})^{-1}\\
		=&P_{13} D_{1}P_{13} P_{13}P_{24}diag\{\tilde{p}_{1},\tilde{p}_{2},\tilde{p}_{3},\tilde{p}_{4}\}P_{24} P_{13}P_{13}D_{1}^{T}P_{13}.
	\end{align*}	
	Hence
	\begin{align*}
		N^{*1}_{3}(T-t,K=0)(\eta=S_{3})&=Q_{3}\hat{B}(T-t)^{-1}N^{1}(T-t,K=0)(\eta=S_{3})(\hat{B}(T-t)^{T})^{-1}Q^{T}_{3}\\
		&=P_{13}Q_{1}P_{13}\hat{B}(T-t)^{-1}N^{1}(T-t,K=0)(\eta=S_{3})(\hat{B}(T-t)^{T})^{-1}P_{13}Q^{T}_{1}P_{13}\\
		&=diag\{\tilde{p}_{3},\tilde{p}_{4},0,\tilde{p}_{2}\}=P_{13}P_{24}N^{*1}_{1}(T-t,K=0)(\eta=S_{1})P_{24}P_{13}.
	\end{align*}
	Similarly, due to definition \eqref{4.13}
	\begin{align*}
		M^{*1}_{3}(T-t,K=0)(\eta=S_{3})&=Q_{3}\hat{B}(T-t)^{-1}M^{1}(T-t,K=0)(\eta=S_{3})(\hat{B}(T-t)^{T})^{-1}Q^{T}_{3},
	\end{align*}
	Thanks to Lemma \ref{lem:4.1} and equation \eqref{eq:4.12}, an argument similar to the one used in Lemma \ref{lem:4.5} shows that
	\begin{align*}
		M^{*1}_{3}(T-t,K=0)(\eta=S_{3})&=P_{13}P_{24}M_{1}^{*1}(T-t,K=0)(\eta=S_{1})P_{13}P_{24}.
	\end{align*}
	Finally, owing to
	\begin{align*}
		I_{k}&=P_{1k}P_{ij}I_{1}P_{1k}P_{ij},
	\end{align*}	
	Hence
	\begin{align*}
		N^{*1}_{3}(T-t,K=0)(\eta=S_{3})-c I_{k}&=P_{13}P_{24}(N_{1}^{*1}(T-t,K=0)(\eta=S_{1})-c I_{1})P_{13}P_{24},\\
		M^{*1}_{3}(T-t,K=0)(\eta=S_{3})-c I_{k}&=P_{13}P_{24}(M_{1}^{*1}(T-t,K=0)(\eta=S_{1})-c I_{1})P_{13}P_{24}.
	\end{align*}
	This completes the proof of Lemma \ref{lem:4.11}.
\end{proof}

\begin{theorem}\label{thm:5.12}
	Consider a class of two-points function $f$, let $F(t,\eta)=P_{T-t}f(\eta)$. For any time $T-t$, there exists a constant $\rho(T-t,K)>0$ which is independent to $f$ and initial data $\eta$, such that 
	\begin{align}\label{eq:13}
		\Gamma_{2}(F,F)(\eta)\geq \frac{\rho(T-t,K)}{2}\Gamma(F,F)(\eta),
	\end{align}
	where $\eta\in\{+1,-1\}^{\mathbb{Z}}$. $\rho(T-t,K)$ is continuous about $K$ and $\rho(T-t,K=0)=\rho$. Furthermore, coupling system parameter $K$ with time parameter $t$ which satisfies 
	$$K_{t}\rightarrow K^{*}= 0, \ \text{and} \ \rho(t,K_{t})\rightarrow \rho$$
	as $t\rightarrow \infty$. Then there exists $T^{*}$ large enough, such that for $T\geq T^{*}$ , the local Poincar\'e inequality holds for any two-points function $f$
	\begin{align*}
		\mathcal{P}_{T}\left(f^{2}\right)(\eta)-(\mathcal{P}_{T}\left(f\right))^{2}(\eta) 
		&\leq 2\int_{0}^{T^{*}}e^{-\int_{0}^{t}\rho(s,K_{s}) ds}dt\cdot \mathcal{P}_{T}\left(\Gamma(f,f)\right)(\eta)\\
		+&2(\frac{2}{\rho}-\frac{2}{\rho}e^{\frac{-\rho(T-T^{*})}{2}})e^{-\int_{0}^{T^{*}}\rho(s,K_{s}) ds}\cdot \mathcal{P}_{T}\left(\Gamma(f,f)\right)(\eta),
	\end{align*} 
	where $\eta\in E$ is the initial data. Finally, if the semi-group $\mathcal{P}_{t}$ is ergodic in the sense that
	\begin{align*}
		\mu(f)=\lim\limits_{T\rightarrow \infty}(\mathcal{P}_{T}f)(\eta),
	\end{align*}
	where $\mu$ denote the probability measure and $\eta$ is an arbitrary initial data. Then,  as the system time  $T$ goes to $ +\infty$, the  inequality for limit measure $\mu$ follows
	\begin{align}\label{eq:4.12.1}
		\mu\left(f^{2}\right)-(\mu\left(f\right))^{2}\leq \frac{2}{\rho} \mu\left(\Gamma(f,f)\right).
	\end{align}
\end{theorem}
\begin{proof}
	Let $\pi_{x}(\eta)$  be the projection at coordinate $x$ of configuration $\eta$, which means $\pi_{x}(\eta)=\eta(x)$ $x\in \mathbb{Z}$. Dividing the configuration space into four classes according to spatial  coordinate $x_{1}$ and $x_{2}$ 
	\begin{align*}
		\bar{\eta}_{1}&=\{\eta \in \{+1,-1\}^{Z}|\pi_{x_{1}}(\eta)=\eta(x_{1})=1,\pi_{x_{2}}(\eta)=\eta(x_{2})=1\},\\
		\bar{\eta}_{2}&=\{\eta \in \{+1,-1\}^{Z}|\pi_{x_{1}}(\eta)=\eta(x_{1})=-1,\pi_{x_{2}}(\eta)=\eta(x_{2})=1\},\\
		\bar{\eta}_{3}&=\{\eta \in \{+1,-1\}^{Z}|\pi_{x_{1}}(\eta)=\eta(x_{1})=1,\pi_{x_{2}}(\eta)=\eta(x_{2})=-1\},\\
		\bar{\eta}_{4}&=\{\eta \in \{+1,-1\}^{Z}|\pi_{x_{1}}(\eta)=\eta(x_{1})=-1,\pi_{x_{2}}(\eta)=\eta(x_{2})=-1\}.
	\end{align*}
	Denote
	$$\rho_{i}(T-t,K)(\eta)=\frac{\lambda^{M_{i}^{*1}(T-t,K)(\eta)-\varepsilon I_{i}}_{2}}{\lambda^{N_{i}^{*1}(T-t,K)(\eta)+\varepsilon I_{i}}_{4}},\eta\in \bar{\eta_{i}},t\in[0,T],$$
	where
	\begin{align*}
		I_{1}=diag\{0,1,1,1\}, I_{2}=diag\{1,0,1,1\},
		I_{3}=diag\{1,1,0,1\}, I_{4}=diag\{1,1,1,0\}.
	\end{align*}
	Specifically, according to proposition \ref{pro:4.10}, under the circumstance of $K=0$, for $\varepsilon>0 $ sufficiently small, we have
	$$\rho=\frac{\lambda^{M_{1}^{*1}(T-t,K=0)(\eta=S_{1})-\varepsilon I_{1}}_{2}}{\lambda^{N_{1}^{*1}(T-t,K=0)(\eta=S_{1})+\varepsilon I_{1}}_{4}}=\frac{\lambda^{M_{i}^{*1}(T-t,K=0)(\eta=S_{i})-\varepsilon I_{i}}_{2}}{\lambda^{N_{i}^{*1}(T-t,K=0)(\eta=S_{i})+\varepsilon I_{i}}_{4}}>0.$$  
	Given the fixed time $T-t$ and initial data $\eta$, elements of $M^{1}(T-t,K)(\eta)$ and $N^{1}(T-t,K)(\eta)$ are continuous about $K$, because
	\begin{align*}
		M_{i}^{*1}(T-t,K)(\eta)&=Q_{i}\hat{B}(T-t)^{-1}M^{1}(T-t,K)(\eta)(Q_{i}\hat{B}(T-t)^{-1})^{T},\\
		N_{i}^{*1}(T-t,K)(\eta)&=Q_{i}\hat{B}(T-t)^{-1}N^{1}(T-t,K)(\eta)(Q_{i}\hat{B}(T-t)^{-1})^{T}.
	\end{align*}
	Hence, the eigenvalues of $M_{i}^{*1}(T-t,K)(\eta)$ and $N_{i}^{*1}(T-t,K)(\eta)$ are continuous about $K$. Define
	\begin{align}\label{eq:4.14}
		\rho(T-t,K):= \min_{i=1,2,3,4}  \min_{\eta \in E}  \rho_{i}(T-t,K)(\eta).
	\end{align}
	According to proposition \ref{pro:4.10},  $M_{i}^{*1}(T-t,K)(\eta)$ and $N_{i}^{*1}(T-t,K)(\eta)$ only depend on the finite dimensional projection of configuration $\eta$, so there are only finite terms to choose the minimum one in ~\eqref{eq:4.14}~. Hence, for fixed time $T-t$, $\rho(T-t,K)$ is continuous about $K$, and obviously, $\rho(T-t,K=0)=\rho$.
	
	Strictly speaking, for arbitrary time $T-t$, there exists $\varepsilon>0 $ small enough and  $0< \delta_{T-t}$ such that for $\forall \ |K|<\delta_{T-t}$ 
	$$\rho(T-t,K)\geq \frac{\rho}{2}>0.$$
	Because
	\begin{align*}
		M_{i}^{*}(T-t,K)(\eta)-\rho(T-t,K) N_{i}^{*}(T-t,K)(\eta)
		&=M_{i}^{*1}(T-t,K)(\eta)-\rho(T-t,K) N_{i}^{*1}(T-t,K)(\eta)\\
		&+M_{i}^{*2}(T-t,K)(\eta)-\rho(T-t,K) N_{i}^{*2}(T-t,K)(\eta).
	\end{align*}
	and according to proposition \ref{pro:4.10}
	\begin{align*}
		-\varepsilon \cdot I_{i}\leq M_{i}^{*2}(T-t,K)(\eta)&\leq \varepsilon \cdot I_{i},\\
		-\varepsilon \cdot I_{i}\leq N_{i}^{*2}(T-t,K)(\eta)&\leq \varepsilon \cdot I_{i},i=1,2,3,4.
	\end{align*}
	Then
	\begin{align*}
		M_{i}^{*}(T-t,K)(\eta)&-\rho(T-t,K) N_{i}^{*}(T-t,K)(\eta)\\
		&\geq M_{i}^{*1}(T-t,K)(\eta)-\varepsilon \cdot I_{i}-\rho(T-t,K) N_{i}^{*1}(T-t,K)(\eta)-\rho(T-t,K)\varepsilon \cdot I_{i}\\
		&= M_{i}^{*1}(T-t,K)-\varepsilon \cdot I_{i}-\rho(T-t,K) (N_{i}^{*1}(T-t,K)(\eta)+\varepsilon \cdot I_{i}).
	\end{align*}
	By the definition of $\rho(T-t,K)$, we can know
	\begin{align*}
		\rho(T-t,K)\leq  \rho_{i}(T-t,K)(\eta)=\frac{\lambda^{M_{i}^{*1}(T-t,K)(\eta)-\varepsilon I_{i}}_{2}}{\lambda^{N_{i}^{*1}(T-t,K)(\eta)+\varepsilon I_{i}}_{4}}, \ \forall \eta \in E,\ i=1,2,3,4.
	\end{align*}
	Hence
	\begin{align*}
		M_{i}^{*}(T-t,K)(\eta)-\rho(T-t,K) N_{i}^{*}(T-t,K)(\eta)\geq0.
	\end{align*}
	Finally
	\begin{align*}
		0&\leq \hat{B}(T-t)Q_{i}^{-1}(M_{i}^{*}(T-t,K)(\eta)-\rho(T-t,K) N_{i}^{*}(T-t,K)(\eta))(\hat{B}(T-t)Q_{i}^{-1})^{T}\\
		&=\hat{B}(T-t)Q_{i}^{-1}Q_{i}\hat{B}(T-t)^{-1} \cdot (M(T-t,K)(\eta)-\rho(T-t,K)N(T-t,K)(\eta))\\
		&\cdot (Q_{i}\hat{B}(T-t)^{-1})^{T}(\hat{B}(T-t)Q_{i}^{-1})^{T}\\
		&=M(T-t,K)(\eta)-\rho(T-t,K) N(T-t,K)(\eta).
	\end{align*}
	that is to say, for arbitrary  $T-t$, there exists $0< \delta_{T-t}$ such that for  $\forall \ |K|<\delta_{T-t}$  
	\begin{align}\label{eq:4.15}
		4\Gamma_{2}(F,F)(\eta)\geq 2\rho(T-t,K)\Gamma(F,F)(\eta) ,
	\end{align}
	where $\rho(T-t,K)$ is independent to $\eta$.
	Then, according to Theorem \ref{thm:3.2}, the local Poincar\'e inequality follows
	\begin{align*}
		\mathcal{P}_{T}\left(f^{2}\right)(\eta)-(\mathcal{P}_{T}\left(f\right))^{2}(\eta) \leq 2\int_{0}^{T}e^{-\int_{t}^{T}\rho(T-s,K) ds}dt\cdot \mathcal{P}_{T}\left(\Gamma(f,f)\right)(\eta).
	\end{align*}	
	Furthermore, noticed that for any fixed time $T-t,t\in [0,T]$, $\rho(T-t,K)$ is continuous about $K$ at around $K=0$, which means we can choose a coupling relationship between model parameter $K$ and time parameter $t$, such that $K_{t}\rightarrow 0$ and $\rho(t,K_{t})\rightarrow \rho$ as $t\rightarrow \infty$. Hence, for arbitrary $\delta >0$, there exists  $T^{*}$ sufficiently large, such that  $0<\rho-\delta<\rho(t,K_{t})<\rho+\delta$ when $t\geq T^{*}$ , then using theorem \ref{thm:3.2} again, for large $T$ 
	\begin{align*}
		\mathcal{P}_{T}\left(f^{2}\right)(\eta)-(\mathcal{P}_{T}\left(f\right))^{2}(\eta) \leq& 2\int_{0}^{T}e^{-\int_{0}^{t}\rho(s,K_{s}) ds}dt\cdot \mathcal{P}_{T}\left(\Gamma(f,f)\right)(\eta)\\
		\leq&2\int_{0}^{T^{*}}e^{-\int_{0}^{t}\rho(s,K_{s}) ds}dt\cdot \mathcal{P}_{T}\left(\Gamma(f,f)\right)(\eta)\\
		+&2\int_{T^{*}}^{T}e^{-\int_{0}^{T^{*}}\rho(s,K_{s}) ds}e^{-\frac{\rho}{2}(t-T^{*})}dt\cdot \mathcal{P}_{T}\left(\Gamma(f,f)\right)(\eta)\\
		=&2\int_{0}^{T^{*}}e^{-\int_{0}^{t}\rho(s,K_{s}) ds}dt\cdot \mathcal{P}_{T}\left(\Gamma(f,f)\right)(\eta)\\
		+&2(\frac{2}{\rho}-\frac{2}{\rho}e^{\frac{-\rho(T-T^{*})}{2}})e^{-\int_{0}^{T^{*}}\rho(s,K_{s}) ds}\cdot \mathcal{P}_{T}\left(\Gamma(f,f)\right)(\eta).
	\end{align*}
	Let $T\rightarrow \infty$ firstly under the ergodic assumption, we can get
	\begin{align*}
		\mu\left(f^{2}\right)-(\mu\left(f\right))^{2} \leq&2(\int_{0}^{T^{*}}e^{-\int_{0}^{t}\rho(s,K_{s}) ds}dt+\frac{2}{\rho}e^{-\int_{0}^{T^{*}}\rho(s,K_{s}) ds})\cdot \mu\left(\Gamma(f,f)\right).
	\end{align*}
	Since $T^{*}$ is large enough, we get finally
	\begin{align*}
		\mu\left(f^{2}\right)-(\mu\left(f\right))^{2} \leq&\frac{2}{\rho} \mu\left(\Gamma(f,f)\right).
	\end{align*}
	This completes the theorem.
\end{proof}
Particularly, let two-points function to be $f(\omega)=f(\omega_{x_{1}},\omega_{x_{2}})=\omega_{x_{1}}+\omega_{x_{2}}$, we can get the estimate of correlation functions as follows:
\begin{corollary}
	Consider a class of two-points function $f(\omega)=\omega_{x_{1}}+\omega_{x_{2}}$, there exists a relationship between system parameter $K$ and time parameter $t$, which satisfies 
	$$K_{t}\rightarrow K^{*}= 0, \ \text{and} \ \rho(t,K_{t})\rightarrow \rho$$
	as $t\rightarrow \infty$. Then there exists $T^{*}$ large enough, such that for $T\geq T^{*}$ , the following estimate of correlation function holds
	\begin{align*}
		2Cov_{\mathcal{P}_{T}}\left(\omega_{x_{1}};\omega_{x_{2}}\right)(\eta)&\leq Var_{\mathcal{P}_{T}}\left(\omega_{x_{1}}+\omega_{x_{2}}\right)(\eta)\\ &\leq2\int_{0}^{T^{*}}e^{-\int_{0}^{t}\rho(s,K_{s}) ds}dt\cdot \mathcal{P}_{T}\left(\Gamma(f,f)\right)(\eta)\\
		+&2(\frac{2}{\rho}-\frac{2}{\rho}e^{\frac{-\rho(T-T^{*})}{2}})e^{-\int_{0}^{T^{*}}\rho(s,K_{s}) ds}\cdot \mathcal{P}_{T}\left(\Gamma(f,f)\right)(\eta),
	\end{align*} 
	where $\eta\in E$ is the initial data. Finally, under the conditions, including coupling relationship between model parameter $K$ and time parameter $t$, in Theorem \ref{thm:5.12}. we obtain the estimate of correlation function as the system time  $T$ goes to $ +\infty$
	\begin{align*}
		2Cov_{\mu}\left(\omega_{x_{1}};\omega_{x_{2}}\right)\leq Var_{\mu}\left(\omega_{x_{1}}+\omega_{x_{2}}\right)\leq\frac{2}{\rho} \mu\left(\Gamma(f,f)\right).
	\end{align*} 
	\begin{proof}
		Noticed that for any probability measure $\mu$ and random variables $X,Y$
		\begin{align*}
			2Cov_{\mu}\left(X;Y\right)= Var_{\mu}\left(X+Y\right)-Var_{\mu}\left(X\right)-Var_{\mu}\left(Y\right)\leq Var_{\mu}\left(X+Y\right).
		\end{align*} 
		Taking the two-points function as $f(\omega)=\omega_{x_{1}}+\omega_{x_{2}}$, according to Theorem \ref{thm:5.12}, we can obtain directly
		\begin{align*}
			2Cov_{\mathcal{P}_{T}}\left(\omega_{x_{1}};\omega_{x_{2}}\right)(\eta)&\leq Var_{\mathcal{P}_{T}}\left(\omega_{x_{1}}+\omega_{x_{2}}\right)(\eta)\\ &\leq2\int_{0}^{T^{*}}e^{-\int_{0}^{t}\rho(s,K_{s}) ds}dt\cdot \mathcal{P}_{T}\left(\Gamma(f,f)\right)(\eta)\\
			+&2(\frac{2}{\rho}-\frac{2}{\rho}e^{\frac{-\rho(T-T^{*})}{2}})e^{-\int_{0}^{T^{*}}\rho(s,K_{s}) ds}\cdot \mathcal{P}_{T}\left(\Gamma(f,f)\right)(\eta).
		\end{align*} 
		The remainder of argument is similar to what we have done in Theorem \ref{thm:5.12},
		which completes the proof of corollary.
	\end{proof}	
\end{corollary}

\section{Acknowledgement}
This work was supported by National Natural Science Foundation of China (Grant No.12288201).

\section{Appendix}

\begin{appendix}
	\section{Proof of Lemma \ref{lem:A1}}
	\begin{lemma}\label{lem:A1}
		Consider function $f(x)$ which is decreasing on $[a,+\infty)$ for some constant $a$ and 
		\begin{align*}
			\int_{a}^{+\infty} f(x)dx,
		\end{align*}
		is convergent.
		Then
		$$
		\lim_{x\rightarrow +\infty}xf(x)=0.
		$$
	\end{lemma}
	\begin{proof}
		Without loss of generality, let $a\geq 0$. The result will be proved by following argument:\\
		\smallskip\textbf{case 1. $f(x)<0$ on $[a,+\infty)$:} Because $f(x)$ is decreasing,  then $f(x)\leq f(a)<0$, which is contrary to the convergence of
		\begin{align*}
			\int_{a}^{+\infty} f(x)dx.
		\end{align*}
		\smallskip\textbf{case 2. $f(x)\geq0$ on $[x_{0},+\infty)$ for some $x_{0}\geq a$ but $
			\lim_{x\rightarrow +\infty}f(x)\neq0
			$:} Because  
		\begin{align*}
			\int_{a}^{+\infty} f(x)dx,
		\end{align*}
		is convergent, so is
		\begin{align*}
			\int_{x_{0}}^{+\infty} f(x)dx=  \int_{a}^{+\infty} f(x)dx-\int_{a}^{x_{0}} f(x)dx.
		\end{align*}
		Besides, we claim:
		for $\forall \ \varepsilon>0$, $\exists\  x^{*}\geq x_{0}$ large enough such that for any $x\geq x^{*}$
		$$
		f(x)\leq \varepsilon.
		$$
		If not, there exists $\varepsilon_{0}>0$, such that $\forall\  x^{*}\geq x_{0}$, $\exists\  \hat{x}\geq x^{*}$ and $f(\hat{x})\geq \varepsilon$, 
		because $f(x)$ is decreasing, so $f(x)\geq \varepsilon_{0}$ on $[x_{0},)$, then
		\begin{align*}
			\int_{a}^{+\infty} f(x)dx\geq \int_{x_{0}}^{+\infty} f(x)dx\geq \int_{x_{0}}^{\hat{x}} f(x)dx\geq f(\hat{x})(\hat{x}-x_{0})\geq \varepsilon_{0}(\hat{x}-x_{0})\geq \varepsilon_{0}(x^{*}-x_{0}).
		\end{align*}
		Owing to the arbitrary of $x^{*}$, then
		\begin{align*}
			\int_{a}^{+\infty} f(x)dx=+\infty.
		\end{align*}
		This leads to a contradiction.\\
		\smallskip\textbf{case 3. $f(x)\geq0$ on $[x_{0},+\infty)$ for some $x_{0}\geq a$ and $
			\lim_{x\rightarrow +\infty}f(x)=0
			$:}
		By Cauthy convergent theorem, for $\forall \ \varepsilon>0$, $\exists\  x^{*}\geq x_{0}$ large enough such that $f(x^{*})>0$ and for any $x\geq x^{*}$,
		\begin{align*}
			\int_{x^{*}}^{x} f(t)dt\leq \varepsilon,
		\end{align*}
		According to the fact that $f(x)$ is decreasing, then 
		\begin{align*}
			f(x)(x-x^{*})\leq \int_{x^{*}}^{x} f(t)dt\leq \varepsilon,
		\end{align*}
		that is to say
		\begin{align*}
			x^{*}f(x)\leq xf(x)\leq \varepsilon +x^{*}f(x),
		\end{align*}
		Let $x$ goes to $+\infty$, we obtain
		\begin{align*}
			0\leq \lim_{x\rightarrow +\infty}xf(x)\leq \varepsilon,
		\end{align*}
		Because $\varepsilon$ is arbitrary, then
		\begin{align*}
			\lim_{x\rightarrow +\infty}xf(x)=0,
		\end{align*}
		which completes the proof of lemma.
	\end{proof}

	\section{The disaster emerges if using the $\Gamma$ calculus directly}
	\begin{observation}\label{obs:1}
		If we use he $\Gamma$ calculus directly for $F=\mathcal{P}_{T-t}f$, the dimension of matrix that  we need to estimate the eigenvalues can be arbitrarily large.
	\end{observation}
	\begin{proof}
		Consider the two-points function $f$, as we have talked in Section \ref{sec:3}, 	the main estimate in establishing  the local Poincar\'e  inequality is
		\begin{align}\label{eq:B1}
			\Gamma_{2}(F,F)(\eta )&\geq \rho(t) \Gamma(F,F)(\eta ),
		\end{align}
		where $F(t,\eta)=\mathcal{P}_{T-t}f(\eta )$, and
		\begin{align}\label{eq:B2}
			\Gamma(F,F)(\eta)&=\int_{\omega\in E}(F(\omega )-F(\eta))^{2}p(\eta,d\omega ),\nonumber\\
			\Gamma_{2}(F,F)(\eta)&=\int_{\omega\in E}\int_{z\in E}(F(z)-2F(\omega )+F(\eta))^{2}p(\omega,dz )p(\eta,d\omega )\\
			&-2\{\int_{\omega\in E}F^{2}(\omega )p(\eta,d\omega )-\left(\int_{\omega\in E}F(\omega)p(\eta,d\omega )\right)^{2}\}.\nonumber
		\end{align}
		Because $f$ is a two-points function, $f(\eta )=f(\eta_{x_{1}}, \eta_{x_{2}})$, 	where $\eta_{x_{1}}=\Pi_{x_{1}}\eta$. Then for integer time $t$ and $T$, 
		$$ F(t,\eta)=\int_{\omega} f(\omega)p(T-t,\eta,d\omega)=\int_{\omega_{x_{1}}=\pm1,\omega_{x_{2}}=\pm1} f(\omega_{x_{1}},\omega_{x_{2}})p(T-t,\eta,d\omega)$$
		that is to say, $ F(t,\eta)$ is no longer a two-points function rather than a multi-points function depends on at least $2(T-t+1)$ spatial points. In view of the expression in~\eqref{eq:B2}~, $\Gamma(F,F)(\eta)$ and $\Gamma_{2}(F,F)(\eta)$ depend on at least $2(T-t+3)$ spatial points. Denote $2(T-t+3)$ as $N$, for $z,\omega,\eta \in \{+1,-1\}^{N}$, expressing $\Gamma_{2}(F,F)$ and  $\Gamma(F,F)$ as the quadratic form of vector $\vec{F}$, the elements of $\vec{F}$ are $\{F(z_{1},\cdots,z_{N}),z_{i}=\pm1,i=1,\cdots,N\}$ in a binary sort, similar to what we do in Section \ref{sec:3}. For given initial data $\eta$, the corresponding probability values of $(F(z)-2F(\omega )+F(\eta))^{2}$ terms in~\eqref{eq:B2}~are $a_{i}, \ i=1,\dots 2^{N}\cdot 2^{N}$ and the variance terms in~\eqref{eq:B2}~are $v_{k},\ k=1,\dots 2^{N}$, then $\Gamma_{2}$ can be written as  
		\begin{align*}
			\Gamma_{2}=\vec{F}^{T}(\sum_{k=1}^{2^{N}}I_{k}A_{k}I^{T}_{k}-2V)\vec{F}
		\end{align*}
		For convenience, we omit $\vec{F}$ and denote
		\begin{align*}
			\Gamma_{2}=\sum_{k=1}^{2^{N}}I_{k}A_{k}I^{T}_{k}-2V
		\end{align*}
		Obviously, zero is a eigenvalue of this quadratic form, which corresponds to eigenvector $(1,1,\dots ,1)$. Then we consider the matrix gets rid of the first row and first column
		\begin{align*}
			\Gamma^{*}_{2}=\sum_{k=1}^{2^{N}}P_{k}-2V^{*}
		\end{align*}
		where
		\begin{align*}
			P_{1}=\left(\begin{array}{cccccccccc} a_{2}& 0& 0& 0\dots &0\\
				0  &a_{3}&0& 0\dots&0\\
				0 &0 &a_{4}& 0\dots&0\\
				\vdots & && \ddots  &\vdots\\
				0 &0 &0& 0\dots&a_{2^{N}}\end{array}\right), V^{*}=\left(\begin{array}{cccccccccc}  v_{2}-v^{2}_{2}& v_{2}v_{3}& v_{2}v_{4}& v_{2}v_{5}\dots & v_{2}v_{2^{N}}\\
				v_{2}v_{3}  &v_{3}-v^{2}_{3}&v_{3}v_{4}& v_{3}v_{5}\dots&v_{3}v_{2^{N}}\\
				v_{2}v_{4} &v_{3}v_{4}&v_{4}-v^{2}_{4}& v_{4}v_{5}\dots&v_{4}v_{2^{N}}\\
				\vdots & && \ddots  &\vdots\\
				v_{2}v_{2^{N}} &v_{3}v_{v_{2^{N}}} &v_{4}v_{2^{N}}& v_{5}v_{2^{N}}\dots&v_{2^{N}}-v^{2}_{2^{N}}\end{array}\right).
		\end{align*}
		For $k\geq 2$
		\begin{align*}
			P_{k}=\left(\begin{array}{cccccccccc} a_{2^{N}(k-1)+2}&&&-2a_{2^{N}(k-1)+2} & \\
				& \ddots\\
				& &a_{2^{N}(k-1)+k-1}&\vdots \\
				-2a_{2^{N}(k-1)+2}&\dots & &D^{k}_{22}&\dots &&-2a_{2^{N}\cdot k}\\
				&&&&a_{2^{N}(k-1)+k+1}\\
				&& &\vdots & &\ddots  \\
				&&&-2a_{2^{N}\cdot k}&& &a_{2^{N}\cdot k}\end{array}\right),
		\end{align*}
		where
		$$D^{k}_{22}=4\sum_{i=1,i\not=k}^{2^{N}}a_{2^{N}(k-1)+i} + a_{2^{N}(k-1)+k}$$
		Hence
		\begin{align*}
			\Gamma^{*}_{2}&=\sum_{k=1}^{2^{N}}P_{k}-2V^{*}\\
			&=\left(\begin{array}{cccccccccc} D^{2}_{22}+\sum_{i=1,i\not=2}^{2^{N}}a_{2^{N}(i-1)+2}&&&-2a_{2^{N}+k}-2a_{2^{N}(k-1)+2} & \\
				& \ddots\\
				& &&\vdots \\
				-2a_{2^{N}+k}-2a_{2^{N}(k-1)+2}&\dots & &D^{k}_{22}+\sum_{i=1,i\not=k}^{2^{N}}a_{2^{N}(i-1)+k}&\dots 
				&&&&\\
				&& & \vdots& \ddots  
			\end{array}\right)\\
			&-2\left(\begin{array}{cccccccccc}  v_{2}-v^{2}_{2}& v_{2}v_{3}& v_{2}v_{4}& v_{2}v_{5}\dots & v_{2}v_{2^{N}}\\
				v_{2}v_{3}  &v_{3}-v^{2}_{3}&v_{3}v_{4}& v_{3}v_{5}\dots&v_{3}v_{2^{N}}\\
				v_{2}v_{4} &v_{3}v_{4}&v_{4}-v^{2}_{4}& v_{4}v_{5}\dots&v_{4}v_{2^{N}}\\
				\vdots & && \ddots  &\vdots\\
				v_{2}v_{2^{N}} &v_{3}v_{v_{2^{N}}} &v_{4}v_{2^{N}}& v_{5}v_{2^{N}}\dots&v_{2^{N}}-v^{2}_{2^{N}}\end{array}\right).
		\end{align*}
		Similarly, for $\Gamma(F,F)$, we can get
		\begin{align*}
			\Gamma(F,F)=\vec{F}^{T}\hat{V}\vec{F},\ \text{where}\
			\hat{V}=\left(\begin{array}{cccccccccc} 
				v_{1}\\
				& v_{2}\\
				&&&& \ddots  \\
				&&&&&v_{2^{N}}\end{array}\right).
		\end{align*}
		Thus, if we want to estimate the Bakry-\'Emery curvature $\rho(t)$ in~\eqref{eq:B1}~, the distribution of eigenvalues with respect to  $N$-dimentional matrix $\Gamma^{*}_{2}$ should be investigated. Unfortunately, the properties of $\Gamma^{*}_{2}$  in our model is too awful. Even in the special case $K=0$, despite the diagonal elements of the matrix $\Gamma^{*}_{2}$ are all positive, it is not a diagonally dominant matrix and can not be decomposed into sub-matrices as Cushing et.al do in \cite{David2020document}, because all the elements in $\Gamma^{*}_{2}$ are nonzero.

	\end{proof}
\end{appendix}


\end{document}